\UseRawInputEncoding 
\documentclass[titlepage]{article}
\usepackage{layout}
\usepackage{url}
\usepackage{proof}
\usepackage{mathrsfs}
\usepackage{amsmath,amsthm,amssymb}
\usepackage{braket}
\usepackage[dvipdfmx]{graphicx}
\usepackage[dvipdfmx]{color}
\usepackage{mathtools}
\usepackage[all]{xy}

\theoremstyle{definition}
\newtheorem{theorem}{Theorem}
\newtheorem{proposition}[theorem]{Proposition}
\newtheorem{lemma}[theorem]{Lemma}
\newtheorem{claim}[theorem]{Claim}
\newtheorem{defn}[theorem]{Definition}
\newtheorem{cor}[theorem]{Corollary}
\newtheorem{eg}[theorem]{Example}
\newtheorem{rmk}[theorem]{Remark}
\newtheorem{conj}{Conjecture}
\newtheorem{question}{Question}
\newtheorem{convention}[theorem]{Convention}

\title{On some $\Sigma^{B}_{0}$-formulae generalizing counting principles over $V^{0}$}
\author{Eitetsu Ken\footnote{email: yeongcheol-kwon@g.ecc.u-tokyo.ac.jp}}

\newcommand{\QQ}{\mathbb{Q}}
\newcommand{\NN}{\mathbb{N}}
\newcommand{\ZZ}{\mathbb{Z}}

\newcommand{\FF}{\mathbb{F}}
\newcommand{\KK}{\mathbb{K}}

\DeclareMathOperator{\dom}{dom}
\DeclareMathOperator{\ran}{ran}

\DeclareMathOperator{\LA}{\bold{LA}}
\DeclareMathOperator{\LAP}{\bold{LA}\mathtt{P}}

\DeclareMathOperator{\VNC}{\bold{VNC}}
\DeclareMathOperator{\VTC}{\bold{VTC}}

\begin{document}

\maketitle

\begin{abstract}
We formalize various counting principles and compare their strengths over $V^{0}$.
In particular, we conjecture the following mutual independence between:
\begin{itemize}
\item a uniform version of modular counting principles and the pigeonhole principle for injections,
\item a version of the oddtown theorem and modular counting principles of modulus $p$, where $p$ is any natural number which is not a power of $2$,
\item and a version of Fisher's inequality and modular counting principles.
\end{itemize}
Then, we give sufficient conditions to prove them.
We give a variation of the notion of $PHP$-tree and $k$-evaluation to show that any Frege proof of the pigeonhole principle for injections admitting the uniform counting principle as an axiom scheme cannot have $o(n)$-evaluations.
As for the remaining two, we utilize well-known notions of $p$-tree and $k$-evaluation and reduce the problems to the existence of certain families of polynomials witnessing violations of the corresponding combinatorial principles with low-degree Nullstellensatz proofs from the violation of the modular counting principle in concern.
\end{abstract}

\section{Introduction}
 \quad Ajtai's discovery (\cite{Ajtai}) of $V^{0}\not \vdash ontoPHP^{n+1}_{n}$, where $ontoPHP^{n+1}_{n}$ is a formalization of the statement ``there does not exist a bijection between $(n+1)$ pigeons and $n$ holes,'' was a significant breakthrough in proof complexity.
 The technique, which was later formalized in \cite{remake} as $k$-evaluation and switching lemma, has been utilized to further works in the area such as the comparison between various types of counting principles (such as \cite{countvsphp} and \cite{More}).
 In the course of the works, it turned out that degree lower bounds of Nullstellensatz proofs are essential when one would like to give lower bounds for the lengths of proofs from constant depth Frege system equipped with $\{Count^{p}_{k}\}_{1 \leq k \in \NN}$ (i.e., the modular counting principle mod $p$) as an axiom scheme.
 One of the most important open problems in the current proof complexity is whether $V^{0}(p) \vdash injPHP^{n+1}_{n}$ or not (here, $injPHP^{n+1}_{n}$ denotes the pigeonhole principle for injections). 
 This problem is interesting because it would deepen our understanding of how hard it is to count a set (recall that $\VTC^{0} \vdash injPHP^{n+1}_{n}$ and $V^{0}(p) \vdash Count^{p}_{n}$). 
 Furthermore, if the problem is solved for composite $p$, it would give us tips to solve the separation problem $AC^{0}(p)$ versus $TC^{0}$.
 As for this problem, \cite{Razborov} made considerable progress. 
 The paper gave good degree lower bounds for polynomial calculus proofs of $injPHP^{m}_{n}$ ($m>n$), which does not depend on the specific coefficient field.
 
This paper aims to connect the result of \cite{Razborov} to a superpolynomial lower bound for the proof length of $injPHP^{n+1}_{n}$ from $AC^{0}$-Frege system equipped with \textit{Uniform Counting Principle}, which can be seen as a uniform version of infinite formulae $\{Count^{p}_{n}\}_{p,n}$, as an axiom scheme. 
 Tackling the issue, we obtain natural and exciting open problems.
 We also consider some of them, too.
 
  The detailed content and the organization of the article are as follows.
  
 First, in \S \ref{Preliminaries}, as the preliminary, we define three types of (first-order and propositional) formulae, which at first glance seem to be generalized versions of modular counting principles (which do not fix the modulus). We name them as:
 \begin{enumerate}
  \item \textit{Modular Pigeonhole Principle} $modPHP^{d,m}_{n}$.
  \item \textit{Uniform Counting Principle} $UCP^{l,m}_{n}$.
  \item \textit{Generalized Counting Principle} $GCP$.
 \end{enumerate}
 Then, we compare the relative strength of these versions over $V^{0}$ and also develop some independence results over $V^{0}$.
 It immediately turns out that 
 \begin{itemize}
 \item $V^{0} + ontoPHP^{L}_{l} \vdash modPHP^{q,m}_{k}$, and
 \item $V^{0} + modPHP^{q,m}_{k} \vdash ontoPHP^{L}_{l}$.
 \end{itemize}
 (For the precise meaning of the statements, see \S \ref{Preliminaries})
 
 Therefore, we see $modPHP^{d,m}_{n}$ is actually not appropriate to be called a generalized version of modular counting principles because it cannot imply them.
 
 On the other hand, we observe that 
 \begin{itemize}
  \item For any natural number $p \geq 2$, $V^{0}+UCP^{l,m}_{k} \vdash Count^{p}_{n}$.
  \item $V^{0}+UCP^{l,m}_{k} \vdash ontoPHP^{n+1}_{n}$.
  \item $V^{0}+GCP \vdash UCP^{l,m}_{n}$.
  \item $V^{0}+GCP \vdash injPHP^{n+1}_{n}$.
 \end{itemize}
 Hence, we see $UCP^{l,m}_{n}$ and $GCP$ can be seen as generalizations of counting principles. 
 In the latter sections, we tackle natural questions arising from the observations above.
 Towards them, in \S \ref{Preliminaries}, we review the Nullstellensatz proof system.
 
 In \S \ref{On Question UCPinjPHP}, we consider the problem: $V^{0}+UCP^{l,m}_{k} \vdash injPHP^{n+1}_{n}$? 
 The author conjectures $V^{0}+UCP^{l,m}_{k} \not \vdash injPHP^{n+1}_{n}$, and gives a sufficient condition to prove it. 
 We define suitable analogues of \textit{PHP-tree} and \textit{$k$-evaluation} in the proof of Ajtai's theorem given in \cite{proofcomplexity}, and show that if \textit{an $h$-evaluation using $injPHP$-trees} exists for a Frege proof of $injPHP^{n+1}_{n}$ admitting $UCP^{l,m}_{k}$ as an axiom scheme, $h$ cannot be of order $o(n)$. 
 Our main work is manipulating the trees that connect the order of $h$ with the degrees of Nullstellensatz proofs of it. 
 
In \S \ref{On the strength of the oddtown theorem}, we consider a ``more natural'' combinatorial principle than $GCP$ that implies both $injPHP^{n+1}_{n}$ and $Count^{p}_{n}$ for some $p$. 
 Namely, we consider the (propositional and first-order) formulae $oddtown_{n}$, which formalize the the oddtown theorem. We observe:
 \begin{itemize}
  \item $V^{0}+oddtown_{k} \vdash Count^{p}_{n}$ for $p=2^{l}$.
  \item $V^{0}+oddtown_{k} \vdash injPHP^{n+1}_{n}$.
 \end{itemize}
 The author conjectures $V^{0}+oddtown_{k} \not\vdash Count^{p}_{n}$ for any prime $p \neq 2$, and gives a sufficient condition to prove it.
 Roughly speaking, the statement is as follows; if $V^{0}+oddtown_{k} \vdash Count^{p}_{n}$, then there exists a constant $\epsilon > 0$ such that for each $n$, we can construct a vector of $n^{O(1)}$ many $\FF_{2}$-polynomials whose violating oddtown condition can be verified by Nullstellensatz proofs from $\lnot Count^{p}_{n^{\epsilon}}$ with degree $O(1)$.
 
 In \S \ref{On the strength of Fisher's inequality}, we consider the (propositional and first-order) formulae $FIE_{n}$ which formalize Fisher's inequality.
 We observe
 \begin{itemize}
  \item $V^{0}+FIE_{k} \vdash injPHP^{n+1}_{n}$.
 \end{itemize}
 On the other hand, the author conjectures $V^{0}+FIE_{k} \not \vdash Count^{p}_{n}$ for any $p \geq 2$, and gives a sufficient condition whose form is similar to the previous one to prove it.

\section{Preliminaries}\label{Preliminaries}
\subsection{Setup of bounded arithmetic and counting principles}
\quad Throughout this paper, $p$ and $q$ denote natural numbers. 
The cardinality of a finite set $S$ is denoted by $\# S$.
We prioritize readability and often use natural abbreviations to express logical formulae. We assume the reader is familiar with the basics of bounded arithmetics and Frege systems (such as the concepts treated in \cite{Cook}).
Unless stated otherwise, we follow the conventions of \cite{Cook}.

We basically use the parenthesis $(,)$ to denote tuples, but when there are several tuples of different types or in different universes, we also use $\langle , \rangle$ for readability.
In particular, in \S 5, we often label edges and leaves of a tree by tuples. 
In that case, we use $\langle, \rangle$.

As propositional connectives, we use only $\bigvee$ and $\lnot$. We assume $\bigvee$ has unbounded arity. 
When the arity is small, we also use $\lor$ to denote $\bigvee$. 
We define an abbreviation $\bigwedge$ by
\begin{align*}
\bigwedge_{i=1}^{k} \varphi_{i}:= \lnot \bigvee_{i=1}^{k} \lnot \varphi_{i}.
\end{align*}
 When the arity of $\bigwedge$ is small, we also use $\land$ to denote it.
 We give the operators $\bigvee$ and $\bigwedge$ precedence over $\lor$ and $\land$ as the order of application.
 \begin{eg}
 $\bigwedge_{i} \varphi_{i} \lor \bigwedge_{j} \psi_{j}$ means $(\bigwedge_{i} \varphi_{i}) \lor (\bigwedge_{j} \psi_{j})$.
 \end{eg}
 We also define an abbreviation $\rightarrow$ by 
 \begin{align*}
 (\varphi \rightarrow \psi) := \lnot \varphi \lor \psi.
 \end{align*}
 For a set $S$ of propositional variables, an \textit{$S$-formula} means a propositional formula whose propositional variables are among $S$.
 For a set $S=\{s^{j}_{i} \mid 1 \leq j \leq k, i \in I_{j}\}$ of propositional variables where each $s^{j}_{i}$ is distinct, 
 an $S$-formula $\psi$, and a family $\{\varphi^{j}_{i}\}_{i \in I_{j}}$ ($j = 1, \ldots, k$) of propositional formulae,
 \begin{align*}
  \psi[\varphi^{1}_{i}/s^{1}_{i}, \cdots, \varphi^{k}_{i} / s^{k}_{i}]
 \end{align*}
   denotes the formula obtained by substituting each $\varphi^{j}_{i}$ for $s^{j}_{i}$ simultaneously.
   
It is well-known that a $\Sigma^{B}_{0}$ $\mathcal{L}^{2}_{A}$-formula $\varphi(x_{1}, \ldots,x_{k},R_{1}, \ldots, R_{l})$ can be translated into a family $\{\varphi[n_{1},\ldots, n_{k}, m_{1}, \ldots, m_{l}]\}_{n_{1},\ldots, n_{k}, m_{1}, \ldots, m_{l} \in \NN}$ of propositional formulae. (See Theorem VII 2.3 in \cite{Cook}.)

Now, we define several formulae which express the so-called ``counting principle.''
\begin{defn}
 For each $p \geq 2$, let $Count^{p}(n,X)$ be an $\mathcal{L}^{2}_{A}$-formula as follows (intuitively, it says for $n \not \equiv 0 \pmod p$, $[n]$ cannot be $p$-partitioned):
 \begin{align*}
  Count^{p}(n,X) := \lnot p \mid n \rightarrow \lnot\Bigg(
  &\Big(\forall k \in [n].\exists e \in [n]^{(p)}. (e \in X \land k \in^{*} e)\Big) \\
  &\land \Big(\forall e, e^{\prime} \in [n]^{(p)}. \lnot (e\in X \land e^{\prime} \in X\land e \perp e^{\prime})\Big)\Bigg)
 \end{align*}
 Here, 
 \begin{itemize}
  \item $p \mid n$ is a $\Sigma^{B}_{0}$ formula expressing ``$p$ divides $n$.''
  \item $[n]$ denotes the set $\{1, \ldots,n\}$.
  \item We code a $p$-subset 
 \[e=\{e_{1}< \cdots< e_{p}\}\]
  of $[n]$ by the number 
 \[\sum_{i=1}^{p}e_{i}(n+1)^{p-i}.\]
 \item $[n]^{(p)}$ denotes the $\Sigma^{B}_{0}$-definable set of all the codes of the $p$-subsets of $[n]$.
 Note that each member in $[n]^{(p)}$ is less than, say, $(n+1)^{p}$.

 \item The elementship relation $\in^{*}$ is expressed by a natural $\Sigma^{B}_{0}$-predicate.
 We often write it $\in$, too. 
 \item $e \perp e^{\prime}$ means 
 \[e \neq e^{\prime} \ \mbox{and} \ e \cap e^{\prime} \neq \emptyset,\] 
 and it is  also expressed by a natural $\Sigma^{B}_{0}$-predicate.
 \end{itemize}
 
 \quad We also define the propositional formula $Count^{p}_{n}$ as in \cite{Krajicek}:
 \begin{align*}
  Count^{p}_{n} := 
  \begin{cases}
  1  &(\mbox{if $p \mid n$})\\
  \lnot \left(\bigwedge_{k \in [n]} \bigvee_{e: k \in e \in [n]^{(p)} } r_{e} \land
  \bigwedge_{e,e^{\prime} \in [n]^{(p)} :e \perp e^{\prime}} (\lnot r_{e} \lor \lnot r_{e^{\prime}})\right)  &(\mbox{otherwise}) 
    \end{cases}
 \end{align*}
 Here, $\{r_{e}\}_{e \in [n]^{(p)}}$ is a family of distinct propositional variables.
 \end{defn}

\begin{convention}\label{excuse}
 With suitable identification of propositional variables, $Count^{p}(x,X)[n,(n+1)^{p}]$ is equivalent to $Count^{p}_{n}$ over $AC^{0}$-Frege system modulo polynomial-sized proofs. Thus we often abuse the notation and write $Count^{p}_{n}$ for $Count^{p}(n,X)$.
\end{convention}

\begin{defn}
 The $\Sigma^{B}_{0}$ $\mathcal{L}^{2}_{A}$-formula $ontoPHP(m,n,R)$ is a natural expression of the statement ``If $m>n$, then $R$ does not give a graph of a bijection between $[m]$ and $[n]$,'' in a similar way as $Count^{p}(n,X)$. 
 Similarly, the $\Sigma^{B}_{0}$ $\mathcal{L}^{2}_{A}$-formula $injPHP(m,n,R)$ is a natural expression of the statement ``If $m>n$, then $R$ does not give a graph of an injection from $[m]$ to $[n]$.''\\
 \quad We also define the propositional formulae $ontoPHP^{m}_{n}$ and $injPHP^{m}_{n}$ by
 \begin{align*}
  ontoPHP^{m}_{n}:=
  \begin{cases}
  \lnot\Big(\bigwedge_{i \in [m]} \bigvee_{j \in [n]} r_{ij} \land \bigwedge_{i \neq i^{\prime} \in [m]} \bigwedge_{j \in [n]} (\lnot r_{ij} \lor \lnot r_{i^{\prime}j}) \\
  \quad \land \bigwedge_{j \in [n]} \bigvee_{i \in [m]}r_{ij} 
  \land \bigwedge_{j \neq j^{\prime} \in [n]} \bigwedge_{i \in [m]} (\lnot r_{ij} \lor \lnot r_{ij^{\prime}}) \Big) &(\mbox{if $m>n$})\\
  1  &(\mbox{otherwise})
  \end{cases}
 \end{align*}
 and  
  \begin{align*}
  injPHP^{m}_{n}:=
  \begin{cases}
  \lnot\Big(\bigwedge_{i \in [m]} \bigvee_{j \in [n]} r_{ij} \land \bigwedge_{i \neq i^{\prime} \in [m]} \bigwedge_{j \in [n]} (\lnot r_{ij} \lor \lnot r_{i^{\prime}j}) \\
  \quad \land \bigwedge_{j \neq j^{\prime} \in [n]} \bigwedge_{i \in [m]} (\lnot r_{ij} \lor \lnot r_{ij^{\prime}}) \Big)  &(\mbox{if $m>n$})\\
  1   &(\mbox{otherwise})
  \end{cases}
 \end{align*}
 With reasons similar to the one stated in Convention \ref{excuse}, we abuse the notations and use $ontoPHP^{m}_{n}$ to denote $ontoPHP(m,n,R)$ and $injPHP^{m}_{n}$ to denote $injPHP(m,n,R)$.
\end{defn}

The following are well-known:
\begin{theorem}[\cite{Ajtai}, improved by \cite{remake} and \cite{remake2}]\label{Ajtai}
\begin{align*}
 V^{0} \not \vdash ontoPHP^{n+1}_{n}.
\end{align*}
Here, we adopt the following convention.
\end{theorem}
\begin{convention}
For $\Sigma^{B}_{0}$-formulae $\psi_{1}, \ldots, \psi_{l}$ and $\varphi$, we write
\begin{align*}
V^{0}+\psi_{1}+\cdots +\psi_{l} \vdash \varphi
\end{align*}
to express the fact that the theory $V^{0}\cup\{\forall\forall \psi_{i}\mid i \in [l]\}$ implies $\forall\forall \varphi$. 
Here, $\forall\forall$ means the universal closure.\\
\quad We use different parameters to express concrete $\vec{\psi}$ and $\varphi$ in order to avoid the confusion.
We also use letters $p$ and $q$ for fixed parameters of formulae (which are not universally quantified in the theory).
For example, 
\begin{align*}
V^{0}+ Count^{p}_{k}\not \vdash Count^{q}_{n}
\end{align*} 
means
 \begin{align*}
 V^{0}+\forall k,X. Count^{p}(k,X) \not \vdash \forall n,X.\ Count^{q}(n,X),
 \end{align*}
  while
\[V^{0} \not \vdash UCP^{l,d}_{n}\]
 means
\[V^{0} \not \vdash \forall l,d,n,R.\ UCP(l,d,n,R)\]
 (for the definition of $UCP^{l,d}_{n}$ and $UCP(l,d,n,R)$, see Definition \ref{UCP}).
 
In the former example, note that we have used the different variables $k,n$ to avoid confusion about the variables' dependencies.
\end{convention}

\begin{theorem}[\cite{countvsphp}]\label{charcount}
 For $p,q \geq 2$,
 $V^{0}+Count^{p}_{k} \vdash Count^{q}_{n}$ if and only if $\exists N \in \NN.\ q \mid p^{N}$.
\end{theorem}

\begin{theorem}[\cite{More}]\label{CountpinjPHP}
For any $p \geq 2$,
$V^{0}+ Count^{p}_{k} \not \vdash injPHP^{n+1}_{n}$.
\end{theorem}
Also, the following is a corollary of the arguments given in \cite{Krajicek}:
\begin{theorem}[essentially in \cite{Krajicek}]\label{injPHPCountp}
For all $p \geq 2$,
 $V^{0}+ injPHP^{k+1}_{k} \not \vdash Count^{p}_{n}$.
\end{theorem}
\begin{rmk}
Note that the exact statement Theorem 12.5.7 in \cite{Krajicek} shows is 
\begin{align*}
 V^{0} + ontoPHP^{k+1}_{k} \not \vdash Count^{p}_{n}
\end{align*}
for each fixed $p \geq 2$.
However, with a slight change of the argument, it is easy to see that Theorem \ref{injPHPCountp} actually holds.
\end{rmk}

From now on, we consider several seemingly generalized versions of $Count^{p}_{n}$, which do not fix the modulus $p$, and evaluate their strengths.\\
\quad Naively, the generalized counting principle should be a statement like: ``For any $d \geq 2$ and $n \in \NN$, if $d$ does not divide $n$, then $n$ cannot be partitioned into $d$-sets.''
The following is one of the straightforward formalizations of this statement:
\begin{defn}
The $\Sigma^{B}_{0}$ $\mathcal{L}^{2}_{A}$-formula $modPHP(d,m,n,R)$ is a natural formalization of the statement
``If $m \not \equiv n \pmod d$, then $R$ does not give the graph of a bijection between $[m]$ and $[n]$.''\\
\quad We also define the propositional formulae $modPHP^{d,m}_{n}$ as follows:
\begin{align*}
 &modPHP^{d,m}_{n} := \\
 &\begin{cases}
 \mbox{same as the case of $m>n$ in the definition of $ontoPHP^{m}_{n}$} &(\mbox{if $m \not \equiv n \pmod d$})\\
 1 & (\mbox{otherwise})
 \end{cases}
\end{align*}

With a similar reason as the one given in Convention \ref{excuse}, we abuse the notation and use $modPHP^{d,m}_{n}$ to denote $modPHP(d,m,n,R)$.
\end{defn}

Intuitively, $modPHP^{d,m}_{n}$ expresses ``if $n \not \equiv m \pmod d$ and $m=ds+r$ ($0 \leq r <d$), 
then there does not exist a family $\{S_{i}\}_{i \in [q]}$ of $d$-sets and an $r$-set $S_{0}$ which give a partition of $[n]$.''
However, it does not imply even $Count^{2}_{n}$:
\begin{proposition}
The following hold:
\begin{enumerate}
 \item\label{ontoPHPmodPHP} $V^{0} + ontoPHP^{L}_{l} \vdash modPHP^{d,m}_{k}$.
 \item\label{modPHPontoPHP} $V^{0} + modPHP^{d,m}_{k} \vdash ontoPHP^{L}_{l}$.
\end{enumerate}
In particular, for any $p \geq 2$, $V^{0}+ modPHP^{d,m}_{k} \not \vdash Count^{p}_{n}$.
\end{proposition}

\begin{proof}
As for \ref{ontoPHPmodPHP}, argue in $V^{0}$ as follows: assume $m \not \equiv k \pmod d$, and $R$ gives a bijection between $[m]$ and $[k]$. It easily follows that $m \neq k$, and hence $R$ or $R^{-1}$ violates $ontoPHP^{m}_{k}$ or $ontoPHP^{k}_{m}$.

As for \ref{modPHPontoPHP}, argue in $V^{0}$ as follows: suppose $L>l$ and $R$ gives a bijection between $[L]$ and $[l]$. 
Then $R$ violates $modPHP^{L,L}_{l}$.

The last part follows from Theorem \ref{injPHPCountp}.
\end{proof}

Therefore, $modPHP^{d,m}_{n}$ is actually not a generalization of counting principles over $V^{0}$.

 Next, we consider the following version:
\begin{defn}\label{UCP}
$UCP(l,d,n,R)$ (which stands for \textit{Uniform Counting Principle}) is an $\mathcal{L}^{2}_{A}$ formula defined as follows:
\begin{align*}
(d \geq 1 \land \lnot d \mid n ) \rightarrow \lnot
&\forall i \in [l]. (\forall j\in [d]. \exists e \in [n]. R(i,j,e)  \lor \forall j \in [d]. \lnot \exists e \in [n]. R( i,j,e )) \\
&\land \forall (i,j) \in [l]\times[d]. \forall e \neq e^{\prime} \in [n] (\lnot R( i,j,e) \lor \lnot R( i,j,e^{\prime}))\\ 
&\land \forall (i,j) \neq (i^{\prime},j^{\prime}) \in [l] \times [d]. \forall e \in [n]. (\lnot R(i,j,e ) \lor \lnot R( i^{\prime},j^{\prime},e ))\\
&\land \forall e \in [n]. \exists (i,j) \in [l] \times [d]. R( i,j,e )
\end{align*}
The propositional formula $UCP^{l,d}_{n}$ is defined as follows:
\begin{align*}
UCP^{l,d}_{n}:=
\begin{cases}
\lnot \bigg(\bigwedge_{i=1}^{l}\left( \left(\bigwedge_{j=1}^{d} \bigvee_{e \in [n]} r_{i,j,e} \right) \lor \left( \bigwedge_{j=1}^{d} \lnot \bigvee_{e \in [n]} r_{i,j,e} \right)\right) \\
\land \bigwedge_{(i,j) \in [l] \times [d]} \bigwedge_{e \neq e^{\prime} \in [n]} (\lnot r_{i,j,e} \lor \lnot r_{i,j,e^{\prime}})\\ 
\land \bigwedge_{(i,j) \neq (i^{\prime},j^{\prime}) \in [l] \times [d]} \bigwedge_{e \in [n]} (\lnot r_{i,j,e} \lor \lnot r_{i^{\prime},j^{\prime},e})\\
\land \bigwedge_{e \in [n]} \bigvee_{(i,j) \in [l] \times [d]} r_{i,j,e} \bigg) 
\quad (\mbox{if $n \not \equiv 0 \pmod d$, $d \geq 1$})\\
1 \quad (\mbox{otherwise})
\end{cases}
\end{align*}
As in the previous definitions, we abuse the notation and use $UCP^{l,d}_{n}$ to express $UCP(l,d,n,R)$.
\end{defn}
Intuitively, $UCP^{l,d}_{n}$ states ``if $n \not \equiv 0 \pmod d$, then there does not exist a family $\{S_{i}\}_{i \in [l]}$ which consists of $d$-sets and emptysets which give a partition of $[n]$.'' Each variable $r_{i,j,e}$ reads ``the $j$-th element of $S_{i}$ is $e$.''
 
 We observe the following: 
 \begin{proposition}\label{VTC0provesUCP}
  $\VTC^{0} \vdash UCP^{l,d}_{k}$.
 \end{proposition}
 For a detailed treatment of the bounded arithmetic $\VTC^{0}$, see \S IX.3 of \cite{Cook}.
 \begin{proof}[Proof Sketch.]
 Work in $\VTC^{0}$.
 Suppose a bounded set $R$ violates $UCP^{l,d}_{k}$, that is:
 \begin{itemize}
  \item $k \not\equiv 0 \pmod d$.
  \item $[k]$ is partitioned into $\{S_{i}\}_{i \in [l]}$, where $S_{i}=\{e \in [k] \mid \exists j \in [d].\ R(i,j,e)\}$.
  \item Each $S_{i}$ is either empty or $[d]$-set.
  In the latter case, it is witnessed by the bijection 
  \[B_{i}:=\{( j,e ) \mid R(i,j,e)\}.\]
 \end{itemize}
 Crucial properties of $\VTC^{0}$ we use are:
 \begin{enumerate}
  \item There is a $\Sigma^{B}_{1}$-definable function $\#S$, which counts the cardinality of the input bounded set $S$.
  \item\label{bijimplieseq} $\VTC^{0}$ proves the following: if a bounded set $M$ codes a bijection between two bounded sets $A$ and $B$, then $\#A=\#B$.
 \end{enumerate}
 Armed with these, we can derive a contradiction in the following way: 
 we consider the cardinality of $S=\{( i, j, e ) \in [l] \times [d] \times [k] \mid R(i,j,e)\}$.
 We can show that $\#S_{i}$ is $0$ if it is empty and $d$ otherwise.
 In the latter case, we use the item \ref{bijimplieseq} above.
 Furtheremore, by induction on $i \in [l]$, we see that $\#\bigcup_{i' \leq i} S_{i} \equiv 0 \pmod d$.
 Considering the case $i=l$, we have $\# S \equiv 0 \pmod d$.

 On the other hand, there is a $\Sigma^{B}_{0}$-definable bijection between $S$ and $[k]$:
 \[S \rightarrow [k];\ ( i,j,e) \mapsto e.\]
 The fact that the above map is a bijection follows immediately from the assumption on $R$.
 Hence, using the item \ref{bijimplieseq} again, we have $\#S=k \not\equiv 0 \pmod d$, a contradiction. 
 \end{proof}
 
 \begin{proposition}
 The following hold:
 \begin{enumerate}
  \item\label{UCPCount^{p}} For any $p \geq 2$, $V^{0}+UCP^{l,d}_{k} \vdash Count^{p}_{n}$.
  \item\label{UCPontoPHP} $V^{0}+UCP^{l,d}_{k} \vdash ontoPHP^{m}_{n}$.
 \end{enumerate}
 \end{proposition}
 
 \begin{proof}
  As for \ref{UCPCount^{p}}, argue in $V^{0}$ as follows: suppose $n \not \equiv 0 \pmod p$, and $R$ gives a $p$-partition of $[n]$. Set the family $\{S_{r}\}_{r \in [(n+1)^{p}]}$ by:
 \begin{align*}
 S_{r}:=
 \begin{cases}
  \mbox{the set coded by $r$} &(\mbox{if} \quad r \in R)\\
  \emptyset &(\mbox{otherwise})
 \end{cases}
 \end{align*}
 Then $\{S_{r}\}_{r\in [(n+1)^{p}]}$ indeed violates $UCP^{(n+1)^{p},p}_{n}$.
 
 As for \ref{UCPontoPHP}, argue in $V^{0}$ as follows: suppose $m>n$ and $R$ gives a bijection between $[m]$ and $[n]$. Then, $\{[n]\}$ violates $UCP^{1,m}_{n}$.
 \end{proof}
 
 Hence, $UCP^{l,d}_{n}$ is indeed a generalization of counting principles. It is natural to ask
 \begin{question}\label{UCPinjPHP}
 Does the following hold?: 
 \[V^{0}+ UCP^{l,d}_{k} \vdash injPHP^{n+1}_{n},\]
 or, at least,
 \[V^{0} + ontoPHP^{M}_{m} \vdash injPHP^{n+1}_{n}.\]
 
 \end{question}
 In view of theorem \ref{CountpinjPHP}, the author conjectures the answer to the problem is no. We tackle this issue in section \ref{On Question UCPinjPHP}.
 
 Here, we consider one more generalization of counting principles (which relates to the above problem):
 \begin{defn}
 $GCP(P, Q_{1},Q_{2},R_{1},R_{2}, M_{0},M_{1},M_{2})$ (which stands for \textit{Generalized Counting Principle}) is a $\Sigma^{B}_{0}$ $\mathcal{L}^{2}_{A}$-formula expressing the following statement: bounded sets 
 \begin{align*}
 P, Q_{1},Q_{2},R_{1},R_{2}, M_{0},M_{1},M_{2}
 \end{align*}
  cannot satisfy the conjunction of the following properties:
 \begin{enumerate}
  \item $M_{0}$ codes a bijection between $(P \times Q_{1}) \sqcup R_{1}$ and $(P \times Q_{2}) \sqcup R_{2}$.
  \item $M_{1}$ is an injection from $R_{1}$ to $R_{2}$ such that some element $a \in R_{2}$ is out of its range.
  \item $M_{2}$ is an injection from $R_{2}$ to $P$ such that some element $b \in P$ is out of its range.
 \end{enumerate}
 \end{defn}
 \begin{rmk}
 We can consider the propositional translation of $GCP$ as well as the previous examples $UCP^{l,d}_{n}$, $Count^{p}_{n}$, etc. However, we do not write it down here because we are not using it at this time.
 \end{rmk}
 It is easy to see that:
 \begin{proposition}
 $\VTC^{0} \vdash GCP$.
 \end{proposition}
 
 \begin{proof}[Proof Sketch]
 Analogously with Proposition \ref{VTC0provesUCP}, we can derive an arithmetical contradiction in number sort by applying $\#$ to the bounded sets listed in the definition of $GCP$ and their products.
 \end{proof}
 
 \begin{proposition}
 \begin{enumerate}
  \item $V^{0} +GCP \vdash UCP^{l,d}_{n}$.
  \item $V^{0}+GCP \vdash injPHP^{n+1}_{n}$.
 \end{enumerate}
 \end{proposition}
 \begin{proof}
  Work in $V^{0}+GCP$.
  
  We first prove $UCP^{l,d}_{n}$. Suppose $n \not \equiv 0 \pmod d$, and $\{S_{i}\}_{i \in [l]}$ is a family consisting of $d$-sets and emptysets, and $[n]=\bigsqcup_{i \in [l]}S_{i}$. Set 
  \begin{align*}
  Q:=\{i \in [l] \mid S_{i} \neq \emptyset\}.
  \end{align*}
  Then the partition gives a bijection between $[n]$ and $Q \times [d] \sqcup \emptyset$.
  
  On the other hand, since $n \not \equiv 0 \pmod d$, we can write $n=ds+r$ where $1 \leq r < d$. 
  This gives a natural bijection between $[n]$ and $[d] \times [s] \sqcup [r]$.
  
  Using $\Sigma^{B}_{0}$-COMP (cf. Definition V.1.2 in \cite{Cook}), it is straightforward to construct proper injections from $\emptyset$ to $[r]$ and from $[r]$ to $[d]$. Thus, $GCP$ is violated, which is a contradiction.
  
  We next prove $injPHP^{n+1}_{n}$. Suppose $R$ gives an injection from $[n+1]$ to $[n]$. Then, there is a natural bijection between $[n]$ and $[n+1] \times [1] \sqcup ([n]\setminus \ran R)$.
  
  On the other hand, there is a natural bijection between $[n]$ and $[n+1]\times \emptyset \sqcup [n]$.
  
  It is easy to construct proper injections from $[n]\setminus \ran R$ to $[n]$, and from $[n]$ to $[n+1]$. Thus, $GCP$ is violated, which is a contradiction.
 \end{proof}
 It is natural to ask:
\begin{question}
\begin{enumerate}
 \item\label{UCPGCP} Does the following hold: $V^{0}+ UCP^{l,d}_{k} \vdash GCP$?
 \item\label{morenatural} Is there any other combinatorial principle than $GCP$ which also implies $injPHP^{n+1}_{n}$ and some of $Count^{p}_{n}$? \end{enumerate}
\end{question}
On the item \ref{UCPGCP}, the author conjectures the answer is no (since $GCP$ implies $injPHP^{n+1}_{n}$).

As for the item \ref{morenatural}, we consider the oddtown theorem in \S \ref{On the strength of the oddtown theorem}.

\subsection{Nullstellensatz proof system}
 \quad In the analysis of this paper, Nullstellensatz proofs (written shortly as ``$NS$-proofs'') play an essential role in the arguments.
We set up our terminology on $NS$-proofs and review degree lower bounds for $injPHP$.
\begin{defn}
Let $A$ be a (commutative) ring, and
 $\mathcal{F}$ be a set of multivariate $A$-polynomials. 
 For multivariate polynomials $g_{1}$, $g_{2}$ and $h_{f}$ $(f \in \mathcal{F})$ over $A$, \textit{$\{h_{f}\}_{f \in \mathcal{F}}$ is a $NS$-proof of $g_{1}=g_{2}$ from $\mathcal{F}$} over $A$ if and only if 
 \begin{align*}
 g_{1}-g_{2}=\sum_{f \in \mathcal{F}}h_{f}f.
 \end{align*} 
 Especially, for $a \in A \setminus \{0\}$, a $NS$-proof of $a=0$ from $\mathcal{F}$ is called \textit{a NS-refutation of $\mathcal{F}$}.
 
The \textit{degree} of a $NS$-proof $\{h_{f}\}_{f \in \mathcal{F}}$ is defined by $\max_{f \in \mathcal{F}}(\deg(h_{f})+\deg(f))$. 
Here, we adopt the convention $\deg 0:= -\infty$.
\end{defn}

We are particularly interested in the following family of polynomials, which is another representation of the pigeonhole principle:

\begin{defn}
Let $M, m \in \NN$, and $M>m$.
Given a ring $A$, we define $\lnot injPHP^{M}_{m}[A]$ as the family of the following multivariate polynomials over $A$:
\begin{align}
 &x_{ij}^{2}-x_{ij} \quad (i \in [M], j \in [m])\label{booleanaxiom},\\
 &x_{ij}x_{ij^{\prime}} \quad (i \in [M], j\neq j^{\prime} \in [m]), \label{1st}\\
 &x_{ij}x_{i^{\prime}j} \quad (i\neq i^{\prime} \in [M], j \in [m]),\label{2nd}\\
 &\sum_{j=1}^{m} x_{ij} -1 \quad (i \in [M]),\label{3rd}
\end{align}
(each $x_{ij}$ is distinct.)

Furthermore, we define $\lnot inj^{*}PHP^{M}_{m}[A]$ as the family of the following multivariate polynomials over $A$:
\begin{align*}
 &x_{ij}^{2}-x_{ij}, u_{j}^{2}-u_{j} \quad (i \in [M], j \in [m]),\\
 &x_{ij}x_{ij^{\prime}} \quad (i \in [M], j\neq j^{\prime} \in [m]), \\
 &x_{ij}x_{i^{\prime}j} \quad (i\neq i^{\prime} \in [M], j \in [m]),\\
 &\sum_{j=1}^{m} x_{ij} -1 \quad (i \in [M]),\\
 &\sum_{i=1}^{M} x_{ij} + u_{j} -1 \quad (j \in [m])
\end{align*}
($x_{ij}$ and $u_{j}$ are distinct indeterminates).
\end{defn}

\begin{rmk}
It is easy to see that $\lnot injPHP^{M}_{m}[A]$ is not satisfiable in $A$ if $A$ is a field; the system is a paraphrase of the pigeonhole principle since $x^{2}-x=0$ implies $x=0,1$ in a field.
This unsatisfiablity of $\lnot injPHP^{M}_{m}[A]$ is extended to the case when $A$ is a general nontrivial ring since there always exists a ring homomorphism from $A$ to a field, and all the coefficients of the polynomials in $\lnot injPHP^{M}_{m}[A]$ are among $0, \pm 1$.

Furthermore, if $A$ is nontrivial, then $\lnot injPHP^{M}_{m}[A]$ always has a $NS$-refutation over $A$.
Indeed, the unit of $A$ generates a subring of $A$ which is isomorphic to $\ZZ$ or $\ZZ/m\ZZ$ ($m \geq 2$).
$\lnot injPHP^{M}_{m}[\ZZ]$ has a $NS$-refutation over $\ZZ$ since we have a $NS$-refutation over $\QQ$.
We know that $\lnot injPHP^{M}_{m}[\ZZ/m\ZZ]$ has a $NS$-refutation over $\ZZ/m\ZZ$ when $m$ is prime.
For composite $m$, taking a prime divisor $p$ of it, we have an embedding of $\ZZ/p\ZZ$ into $\ZZ/m\ZZ$ as additive groups.
Since all the coefficients of the polynomials in $\lnot injPHP^{M}_{m}[\ZZ/m\ZZ]$ are among $0, \pm 1$, we can transform a $NS$-refutation $(h_{f})_{f}$ over $\ZZ/p\ZZ$ to a $NS$-refutation over $\ZZ/m\ZZ$ by simply mapping each coefficient of $h_{f}$ by the embedding.  

Note that Nullstellensatz proof system is not always complete for general coefficient rings.
See the appendix of \cite{Bussdesign} for details.
\end{rmk}

We have introduced $\lnot inj^{*}PHP^{M}_{m}[A]$ since it makes the translation of $injPHP$-trees into $NS$-refutations in the proof of Theorem \ref{noeval} simpler, although it is an easy observation that $\lnot injPHP^{M}_{m}$ and $\lnot inj^{*}PHP^{M}_{m}$ are equivalent in the sense of degree of $NS$-refutation:
\begin{proposition}\label{injPHPandinj*PHPareequivalent}
Let $A$ be a ring.
 For $M,m,d \in \NN$ with $M > m$, the following are equivalent:
 \begin{enumerate}
 \item\label{injPHPNS} $\lnot injPHP^{M}_{m}[A]$ has a $NS$-refutation over $A$ of degree $\leq d$.
 \item\label{inj*PHPNS} $\lnot inj^{*}PHP^{M}_{m}[A]$ has a $NS$-refutation over $A$ of degree $\leq d$.
 \end{enumerate}
\end{proposition}
\begin{proof}
 (\ref{injPHPNS}) $\Rightarrow$ (\ref{inj*PHPNS}) follows from $\lnot injPHP^{M}_{m}[A] \subseteq \lnot inj^{*}PHP^{M}_{m}[A]$ as families of polynomials.
 
 We consider the converse.
 Suppose $\lnot inj^{*}PHP^{M}_{m}[A]$ has a $NS$-refutation of degree $\leq d$.
  Applying the substitution $u_{j}:=1-\sum_{i=1}^{M} x_{ij}$ ($j \in [m]$), we obtain a $NS$-refutaion of 
  \[\lnot injPHP^{M}_{m} \cup \left\{\left(1-\sum_{i=1}^{M} x_{ij}\right)^{2}-\left(1-\sum_{i=1}^{M} x_{ij}\right) \mid j\in[m]\right\}\]
   over $A$ without increasing the degree. 
   Furthermore, 
\begin{align*}
\left(1-\sum_{i=1}^{M} x_{ij}\right)^{2}-\left(1-\sum_{i=1}^{M} x_{ij}\right) 
= \sum_{i=1}^{M}(x_{ij}^{2}-x_{ij}) + \sum_{i \neq i' \in [M]} x_{ij}x_{i'j}.
\end{align*}
Hence, we obtain a $NS$-refutation of $\lnot injPHP^{M}_{m}[A]$ over $A$ without increasing the degree. 
\end{proof}

As for the case when $A$ is a field, the following powerful degree lower-bound is well-known:
\begin{theorem}[\cite{Razborov}]\label{PHPNSdegreelowerbound}
Let $M,m \in \NN$ and $M>m$.
Assume $A$ is a field.
There is no $PC$-refutation of $\lnot injPHP^{M}_{m}[A]$ over $A$ with degree $\leq m/2$.
In particular, there is no $ NS$ refutation of $\lnot injPHP^{M}_{m}[A]$ over $A$ with degree $\leq m/2$.
\end{theorem}

 Here, $PC$ stands for ``Polynomial Calculus.'' 
 For the precise treatment of the formal system, see \S 6.2 of \cite{proofcomplexity}.

Towards the discussion in \S \ref{On Question UCPinjPHP}, we want to generalize the above degree lower bound for $A=\ZZ/n\ZZ$, where $n$ is a general integer satisfying $n \geq 2$. 
Precisely speaking, taking a closer look at the proof of Theorem \ref{PHPNSdegreelowerbound}, we can obtain the following generalization:

\begin{cor}[essentially by \cite{Razborov}]\label{Razborovlowerboundforfinitering}
Let $A$ be a finite nontrivial commutative ring.
Let $M,m\in \NN$ satisfy $M>m$.
Then there is no $NS$-refutation of $injPHP^{M}_{m}[A]$ of degree $\leq \frac{m}{2}$.
\end{cor}

We give a proof for completeness.
Below, we follow the presentation of \S 16.2 in \cite{proofcomplexity}.
Fix $M>m$ and a finite nontrivial commutative ring $A$ in the rest of this section.

Let $Maps$ be the set of partial bijections from $[M]$ to $[m]$.
For each $\alpha \in Maps$, set $x_{\alpha} := \prod_{i \in \dom(\alpha)} x_{i\alpha(i)}$.
In particular, $x_{\emptyset}:=1$.
Let $\hat{S} := A[\bar{x}]/I$, where $I$ is the ideal generated by the polynomials (\ref{booleanaxiom}), (\ref{1st}), (\ref{2nd}) in $injPHP^{M}_{m}[A]$.
For $f \in A[\bar{x}]$, $\hat{f}$ denotes the quotient of $f$ in $\hat{S}$.
Furthermore, for a subset $U \subseteq A[\bar{x}]$, set 
\[\hat{U}:=\{\hat{f} \in \hat{S} \mid f \in U\}.\]

\begin{defn}
 For a parameter $t \leq m$, define:
\begin{align*}
Maps_{t}&:=\{\alpha \in Maps \mid \#\alpha \leq t\}.\\
S_{t} &:= \{ f \mid f \in A[\bar{x}],\ \deg(f) \leq t\}.\\
T_{t}&:=\{x_{\alpha} \mid \alpha \in Maps_{t}\}.
\end{align*}
\end{defn}

Note that:
\begin{proposition}
For $t \leq m$,
\begin{enumerate}
 \item $\hat{S}_{t}$ is a free $A$-module of finite rank.
 \item $\hat{T}_{t}$ is a basis of $\hat{S}_{t}$, and $\#T_{t}=\#\hat{T}_{t}$.
\end{enumerate}
\end{proposition}

\begin{proof}
The first item follows from the second item.
We can show the second item as follows. 
The fact that $\hat{T}_{t}$ generates $\hat{S}_{t}$ straightforwardly follows from the definition of $I$ above.
We show that $\hat{T}_{t}$ is $A$-linearly independent and $\#T_{t}=\#\hat{T}_{t}$.
Suppose not.
Then we have
\[\sum_{\alpha \in Maps_{t}} c_{\alpha} x_{\alpha} \in I\]
for some $\{c_{\alpha}\}_{\alpha \in Maps_{t}} \subseteq A$ such that $c_{\alpha} \neq 0$ for at least one $\alpha$.
Take a minimal $\alpha_{0}$ of such, and consider the following substitution:
\[
x_{ij} \mapsto 
\begin{cases}
 1 \quad \mbox{(if $\alpha_{0}(i)=j$)}\\
 0 \quad \mbox{(otherwise)}.
\end{cases}
\]
Then every element in $I$ is evaluated to $0$, and $c_{\alpha}x_{\alpha}$ is evaluated to:
\[ \begin{cases}
 c_{\alpha_{0}} \quad \mbox{(if $\alpha=\alpha_{0}$)}\\
 0 \quad \mbox{(otherwise)}.
\end{cases}
\]
because of the minimality of $\alpha_{0}$.
Thus we have $c_{\alpha_{0}}=0$, a contradiction.
\end{proof}

We aim to show that no $a \in A \setminus \{0\}$ is in $O_{t}$ ($t \leq m/2$) below.
\begin{defn}
For $t \leq m$, let $O_{t}$ be the set of $f \in A[\bar{x}]$ such that $f$ has a degree $\leq t$ $PC$-proof from $injPHP^{M}_{m}[A]$.
\end{defn}

Towards it, the following notion is useful to describe the rewriting procedure of polynomials in concern:
\begin{defn}
Let $\alpha \in Maps$.
A \textit{pigeon dance} of $\alpha$ is the following non-deterministic process:
\begin{itemize}
\item Take the first (i.e. the smallest) pigeon $i_{1} \in \dom(\alpha)$ and move it to any unoccupied hole $j$ such that is larger than $\alpha(i_{1})$.
\item Then take the second smallest pigeon $i_{2} \in \dom(\alpha)$ and move it to any currently unoccupied hole larger than $\alpha(i_{2})$, etc., as long as this is possible.
\end{itemize}
We say \textit{pigeon dance is defined on $\alpha$} if this process can be completed for all the pigeons in $\dom(\alpha)$.
\end{defn}
 
\begin{defn}
Let $Maps^{*}$ be the set of all partial maps $\alpha \colon [M] \rightarrow [m] \sqcup\{0\}$ such that $\alpha$ is injective outside $\alpha^{-1}(0)$.
The partial bijection $\alpha \restriction_{\alpha^{-1}([m])}$ is denoted by $\alpha^{-}$.
\end{defn}
For $\alpha \in Maps^{*}$, extend the notation $x_{\alpha}$ as follows:
\[x_{\alpha} := x_{\alpha^{-}} \times \prod_{i \colon \alpha(i)=0}\left(\sum_{j=1}^{m}x_{ij}-1\right).\]

Note that $x_{\alpha}$ is no longer a monomial in general.
We consider the following subsets of $\hat{S}_{t}$:
\begin{defn}
For $t \leq m$, define:
\begin{enumerate}
 \item $B_{t} \subseteq S_{t}$ as the set of all $x_{\alpha} \in Maps^{*}$ such that $\#\alpha \leq t$ and a pigeon dance is defined on $\alpha^{-}$.
 \item $C_{t}:=B_{t}\setminus T_{t}$ and $\Delta_{t} := B_{t} \cap T_{t}$, that is, $\Delta_{t}$ consists of the monomials $x_{\alpha}$ for $\alpha \in Maps_{t}$ such that a pigeon dance is defined on $\alpha$.
\end{enumerate}
\end{defn}

Now, Corollary \ref{Razborovlowerboundforfinitering} is an immediate consequence of the following lemma since $1 \in \Delta_{t}$:

\begin{lemma}
For $t \leq m/2$,
\[\hat{S}_{t} = \hat{O}_{t} \oplus A \hat{\Delta}_{t},\ \hat{O}_{t}=A\hat{C}_{t}\]
as $A$-modules.
In particular, any non-zero scalar in $A$ does not belong to $O_{t}$.
\end{lemma}

\begin{proof}
First, we show that $\hat{B}_{t}=\hat{C}_{t} \cup \hat{\Delta}_{t}$ spans $\hat{S}_{t}$.
For the monomials 
\[x_{\gamma} = x_{i_{1}j_{1}}\cdots x_{i_{k}j_{k}} \ \mbox{and} \ x_{\delta}= x_{u_{1}v_{1}} \cdots x_{u_{l}v_{l}}\quad (i_{1} < \cdots < i_{k} \ \& \ u_{1} < \cdots < u_{l})\]
from $T_{t}$, define a partial ordering $x_{\gamma} \preceq x_{\delta}$ by the condition that:
\begin{itemize}
 \item either $k < l$, or $k=l$ and for the largest $w$ such that $j_{w} \neq v_{w}$ it holds that $j_{w} < v_{w}$.
\end{itemize}

\begin{claim}\label{rewritingprocedure}
 The monomial $x_{\gamma}$ can be expressed as 
\[\widehat{x_{\gamma}} = \sum_{x_{\alpha} \in X} c_{\alpha}\widehat{x_{\alpha}} + \sum_{x_{\beta} \in Y} c_{\beta}\widehat{x_{\beta}}\]
in $\hat{S}$ for some $X \subseteq C_{t}$ and $Y \subseteq \Delta_{t}$, and non-zero coefficients $c_{\alpha}$'s from $A$. 
(Note that $X \cap Y = \emptyset$ by definition of $C_{t}$ and $\Delta_{t}$.)
Further, for all $x_{\alpha} \in X$ and $x_{\beta} \in Y$, $\dom(\alpha) \cup \dom(\beta) \subseteq \dom(\gamma)$.
\end{claim}

We may assume $i_{1} < \cdots<i_{k}$ in $x_{\gamma}$ above and also $x_{\gamma} \not \in \Delta_{t}$.
We use induction on $\preceq$. 
Assume the claim holds for all terms $\preceq$-smaller than $x_{\gamma}$.
In the following, we omit the symbol $\hat{(\cdot)}$ representing the quotient map for readability.
In $\hat{S}$, rewrite $x_{\gamma}$ as
\begin{align*}
x_{i_{2}j_{2}} \cdots x_{i_{k}j_{k}} &- \sum_{\substack{j_{1}'<j_{1}\\ j_{1}' \not \in \{j_{2}, \ldots, j_{k}\}}} x_{i_{1}j_{1}'} \cdots x_{i_{k}j_{k}} - \sum_{\substack{j_{1}'>j_{1}\\ j_{1}' \not \in \{j_{2}, \ldots, j_{k}\}}}x_{i_{1}j_{1}'} \cdots x_{i_{k}j_{k}} \\
&+ \left(\sum_{j=1}^{m}x_{i_{1}j}-1\right)x_{i_{2}j_{2}}\cdots x_{i_{k}j_{k}}. 
\end{align*}
The first term and the terms in the first summation are all $\preceq$-smaller than $x_{\gamma}$, and their domain is included in $\dom(\gamma)$.
Thus, the statement for them follows by the induction hypothesis.
Moreover, the statement for the first term implies that for the last term.
Note that $x_{\alpha} \in B_{t-1}$ implies $\left(\sum_{j=1}^{m}x_{i_{1}j}-1\right)x_{\alpha} \in AB_{t}+I$.

The collection of all the terms in the second summation can be interpreted as describing all possible moves of $i_{1}$ in the attempted pigeon dance of $\gamma$.
To simulate other steps in all possible pigeon dances, rewrite each of these terms analogously using (\ref{3rd}) for $i_{2}$, and so on.
We have assumed $x_{\gamma} \not \in \Delta_{t}$, that the pigeon dance cannot be completed.
Therefore, the rewriting procedure must eventually produce only terms $\preceq$-smaller than $x_{\gamma}$.
This completes the proof of Claim \ref{rewritingprocedure}.

Next, we establish the linear independence of $\hat{B}_{t}=\hat{C}_{t} \cup \hat{\Delta}_{t}$.
It suffices to show 
\[\#B_{t} \leq \# T_{t} = \#\hat{T}_{t}\]
 since $\hat{B}_{t}$ spans $\hat{S}_{t}$, $\hat{T}_{t}$ is its basis, $\#\hat{B}_{t} \leq \#B_{t}$, and $A$ is finite.
 Note that $\#\hat{B}_{t}=\#B_{t}=\#T_{t}$ follows.
 
 Towards the above inequality, consider another procedure \textit{minimal pigeon dance}, which is described as follows:
 given $\alpha \in Maps^{*}$, it is defined by the following instructions:
 \begin{enumerate}
  \item Put $\alpha_{1}:=\alpha$.
  \item For $i=1,\ldots, M$, if $i \in \dom(\alpha_{i})$, move it to the smallest free hole $j > \alpha(i)$ and let $\alpha_{i+1}$ be the resulting map.
  If $i \not \in \dom(\alpha_{i})$, do nothing and put $\alpha_{i+1}:=\alpha_{i}$.
  \item $D(\alpha)$ is the result of this process applied to $\alpha$.
 \end{enumerate}
 
 The following claims are purely combinatorial results on minimal pigeon dance, and the coefficient ring $A$ does not matter, so we state them without proof.
 See the proof of Lemma 16.2.2 in \cite{proofcomplexity} for concrete proofs.
 
 \begin{claim}
 For all $t \leq m/2$, $D$ is defined on the whole of $B_{t}$.
 \end{claim}
 
 \begin{claim}\label{minimalpigeondanceisinjection}
 $D$ is an injective map from $B_{t}$ into $T_{t}$ for $t \leq m/2$.
 \end{claim}
 
 Claim \ref{minimalpigeondanceisinjection} implies $\# B_{t} \leq \# T_{t}$.
 
 
 It remains to show that $\hat{C}_{t}$ is a basis of $\hat{O}_{t}$.
 We already know that $\hat{C}_{t}$ is $A$-linearly independent, so we focus on proving $\hat{C}_{t}$ spans $\hat{O}_{t}$.
 Clearly $C_{t} \subseteq O_{t}$ and $AC_{t}$ contains all the axioms (\ref{3rd}) in $injPHP^{M}_{m}[A]$ and is closed under the addition rule of $PC$.
 For the closure of the space under the multiplication rule, assume that 
 \[x_{\alpha} \in C_{t},\ s=\deg(x_{\alpha}) < t.\]
 It suffices to show that $x_{ij}x_{\alpha} \in AC_{s+1}+I$, which implies $\widehat{fx_{\alpha}} \in A\hat{C}_{t}$ for each monomial $f$ with $\deg(f) + \deg(x_{\alpha}) \leq t$ by induction on $\deg(f)$ and for each $f \in A[\bar{x}]$ with $\deg(f) + \deg(x_{\alpha}) \leq t$ since $A\hat{C}_{t}$ is closed under addition and scalar product.
 We use induction on $s$.
 By the definition of $C_{t}$, we can write $x_{\alpha}=x_{\beta} (1-\sum_{j=1}^{m}x_{i'j})$ for some $i' \in [M]$.
 As $\deg(x_{\beta}) <\deg(x_{\alpha}) = s$, $x_{ij}x_{\beta} \in AB_{s}+I$ by Claim \ref{rewritingprocedure} if $x_{\beta} \in T_{s-1}$ and by induction hypothesis if $x_{\beta} \in C_{s-1}$.
 Note that the degree does not increase.
 Multiplying $(1-\sum_{j=1}^{m}x_{i'j})$,
 we have $x_{ij}x_{\alpha} \in AC_{s+1}+I$.
\end{proof}

\section{On Question \ref{UCPinjPHP}}\label{On Question UCPinjPHP}
 As for Question \ref{UCPinjPHP}, the author conjectures the following:
 \begin{conj}\label{UCPdoesnotproveinjPHP}
  \begin{align*}
 F_{c}+UCP^{l,d}_{k} \not\vdash_{poly(n)} injPHP^{n+1}_{n}.
 \end{align*}
 \quad Here, for a family $\{\alpha_{\vec{k}}\}_{\vec{k} \in \NN}$ of propositional formulae, $F_{c}+\alpha_{\vec{k}}$ is the fragment of Frege system allowing the formulae with depth $\leq c$ only and admitting $\{\alpha_{\vec{k}}\}_{\vec{k}}$ as an axiom scheme.
Furthermore, $P \vdash_{poly(n)} \varphi_{n}$ means each $\varphi_{n}$ has a $poly(n)$-sized $P$-proof.
 \end{conj}
 If this conjecture is true, then it follows that $V^{0}+UCP^{l,d}_{k} \not \vdash injPHP^{n+1}_{n}$ by the witnessing theorem and the translation theorem.
  In this section, we provide a sufficient condition to prove this conjecture.
 Our strategy is to adapt the proof technique of Ajtai's theorem to the situation.
 We define \textit{$injPHP$-tree}, and a \textit{$k$-evaluation using $injPHP$-tree}, and show that a Frege-proof of $injPHP^{n+1}_{n}$ admitting $UCP^{l,d}_{k}$ as an axiom scheme cannot have $o(n)$-evaluation. Taking the contrapositive, we obtain our sufficient condition.
 
 We start with introducing the following notions:
\begin{defn}
Let $D$ and $R$ be disjoint sets. 
A \textit{partial injection from $D$ to $R$} is a set $\rho$ which satisfies the following:
\begin{enumerate}
 \item Each $x \in \rho$ is either a 2-set having one element from $D$ and one element from $R$, or a singleton contained in $R$.
(In the former case, if $x=\{i,j\}$ where $i \in D$ and $j \in R$, we use a tuple $\langle i,j \rangle$ to denote $x$.  
In the latter case, if $x=\{j\}$ where $j \in R$, then we use 1-tuple $\langle j \rangle$ to denote $x$.)
 \item Each pair $x \neq x^{\prime} \in \rho$ are disjoint.
\end{enumerate}
The 2-sets in a partial injection $\rho$ give a partial bijection from $D$ to $R$. We denote it by $\rho_{bij}$. 
Also, we set $\rho_{sing}:=\rho \setminus \rho_{bij}$.

We define $v(\rho):=\bigcup_{x \in \rho}x$, $\dom(\rho) := v(\rho) \cap D$, and $\ran(\rho):=v(\rho) \cap R$.

 For two partial injections $\rho$ and $\tau$ from $D$ to $R$, 
\begin{enumerate}
 \item $\rho || \tau$ if and only if $\rho \cup \tau$ is again a partial injection.
 \item $\rho \perp \tau $ if and only if $\rho || \tau$ does not hold. In other words, there exist $x \in \rho$ and $y \in \tau$ such that $x \neq y$ and $x \cap y \neq \emptyset$.
 \item $\sigma \tau:= \sigma \cup \tau$.
\end{enumerate}
\end{defn}


Now, we introduce an analogy of $PHP$-trees in this context: 
\begin{defn}
Let $D$ and $R$ be disjoint finite sets.
 An \textit{$injPHP$-tree over $(D,R)$} is a vertex-labeled and edge-labeled rooted tree defined inductively as follows:
 \begin{enumerate}
  \item The tree whose only vertex is its root and has no labels is an $injPHP$-tree over $(D,R)$.
  \item If the root is labeled by ``$i \mapsto ?$'' having $\#R$ children and each of its edges corresponding to each label ``$\langle i,j \rangle$'' ($j \in R$), and the subtree which the child under the edge labeled by ``$\langle i,j \rangle$'' induces is an $injPHP$-tree over $(D \setminus \{i\}, R\setminus \{j\})$, then the whole labeled tree is again an $injPHP$-tree over $(D,R)$.
  \item If the root is labeled by ``$? \mapsto j$'' having $(\#D+1)$ children and each of its edges corresponding to each label ``$\langle i,j \rangle$'' ($i \in D$) and ``$\langle j \rangle$,'' and the subtree which the child under the edge indexed by $\langle i,j \rangle$ induces is an $injPHP$-tree over $(D \setminus \{i\}, R\setminus \{j\})$ while the subtree which the child under the edge labeled by ``$\langle j \rangle$'' induces is an $injPHP$-tree over $(D, R\setminus \{j\} )$, then the whole tree is again an $injPHP$-tree over $(D,R)$.
 \end{enumerate}
 For an $injPHP$-tree $T$, we denote \textit{the height} (the maximum number of edges in its branches) of $T$ by $height(T)$ and the set of its branches by $br(T)$.\\
 \quad The pair $(T, L\colon br(T) \rightarrow S)$ is called a \textit{labeled $injPHP$-tree with label set $S$}. For each label $s \in S$, we set $br_{s}(T) := L^{-1}(s)$.
\end{defn}

\begin{convention}
When $T$ is an $injPHP$-tree over $(D,R)$,
 each branch $b \in br(T)$ naturally gives a partial injection, which is the collection of labels of edges contained in $b$. We often abuse the notation and use $b$ to denote the partial injection given by $b$.
\end{convention}

In the following, if there is no problem, we identify domains having the same size $n$ and denote them $D_{n}$. Similarly, we identify ranges $R$ having the same size $n$ and denote them $R_{n}$.
We assume $D_{m}$ and $R_{n}$ are disjoint for any $m,n \in \NN$.
\begin{defn}
For $m>n$, $\mathcal{M}^{m}_{n}$ denotes the set of all partial injections from $D_{m}$ to $R_{n}$.
\end{defn}
\begin{defn}
Let $\rho \in \mathcal{M}^{m}_{n}$ ($m > n$).
 Let $T$ be an $injPHP$-tree over $(D_{m},R_{n})$.
  We define the \textit{restriction} $T^{\rho}$ as the $injPHP$-tree over $(D_{m} \setminus \dom(\rho), R_{n} \setminus \ran(\rho))$ obtained from $T$ by deleting the edges with label incompatible with $\rho$, contracting the edges whose label are contained in $\rho$ (we leave the label of the child), and taking the connected component including the root of the tree.
\end{defn}

In particular, we are interested in shallow $injPHP$-trees.
The main reason is the following Lemma \ref{numberofbranches}.
\begin{defn}
For $\tau \in \mathcal{M}^{m}_{n}$ such that $\tau || \rho$, we set
\begin{align*}
 \tau^{\rho} := \tau \setminus \rho. 
\end{align*}
\end{defn}

\begin{lemma}\label{numberofbranches}
Let $T$ be an $injPHP$-tree and $\rho \in \mathcal{M}^{m}_{n}$ ($m > n$).
If $height(T) \leq n-\#\rho$, then the map
\[\{b \in br(T) \mid b || \rho\} \rightarrow br(T^{\rho});\ b \mapsto b^{\rho} \]
is bijective.
\end{lemma}

\begin{proof}
First, we observe that $b^{\rho} \in  br(T^{\rho})$ if $b \in br(T)$ satisfies $b||\rho$.
Indeed, no edge in $b$ is deleted by the restriction with $\rho$ since $b || \rho$, although some edges in $b$ may be contracted.

Next, we see the above map $b \mapsto b^{\rho}$ is injective.
Indeed, if $b \neq b' \in br(T)$, $b||\rho$, and $b'||\rho$, then $b \perp b'$.
Let $\pi \in b$ and $\pi' \in b'$ witness $b \perp b'$, that is, one of the following holds:
\begin{itemize}
 \item $\pi=\langle i,j \rangle$, $\pi'=\langle i,j' \rangle$, and $j \neq j'$.
 \item $\pi=\langle i,j \rangle$, $\pi'=\langle i',j \rangle$, and $i \neq i'$.
  \item $\{\pi,\pi'\}=\{\langle i,j \rangle,\langle j \rangle \}$.
\end{itemize}
Since $b||\rho$ and $b'||\rho$, we have $\pi',\pi \not \in \rho$, and therefore $\pi \in b^{\rho}$ and $\pi' \in (b')^{\rho}$ follow.
It implies $b^{\rho} \perp (b')^{\rho}$.

Lastly, we show that $b \mapsto b^{\rho}$ is surjective.
Let $r \in br(T^{\rho})$.
By definition of $T^{\rho}$, there exists a vertex $v$ of $T$ such that $a||\rho$ and $r=a^{\rho}$, where $a$ is the partial injection induced by the path from the root to $v$ in $T$.
Let $v$ be the most distant from the root among such.
It suffices to show that $v$ is actually a leaf of $T$, and therefore $a \in br(T)$.
Suppose $v$ is not a leaf of $T$. 
If $h$ is the height of $v$, then $h < n-\#\rho$ by assumption.
It implies $\# a + \# \rho < n$.
Hence, whatever the label of $v$ is, there exists an edge from $v$ whose label is compatible with $\rho$, contradicting the maximality of $v$.
\end{proof}

\begin{defn}
In the setting of the previous Lemma, we denote the inverse mapping by;
\[br(T^{\rho}) \rightarrow \{b \in br(T) \mid b || \rho\};\ r \mapsto r^{-\rho}. \]
\end{defn}

\begin{defn}
Let $\rho \in \mathcal{M}^{m}_{n}$ ($m>n$). 
For $\{r_{ij}\}_{i \in D_{m}, j\in R_{n}}$-propositional formula $\varphi$ (by a natural identification of variables, we regard each variable $r_{ij}$ is utilized to construct the propositional formula $injPHP^{m}_{n}$), we define \textit{the restriction} $\varphi^{\rho}$ by applying $\varphi$ the following partial assignment: for each $i \in [m]$ and $j \in [n]$,
\begin{align*}
 r_{ij} \mapsto 
 \begin{cases}
  1 &\quad (\mbox{if $\langle i,j \rangle \in \rho$})\\
  0 &\quad (\mbox{if $\{\langle i,j \rangle\} \perp \rho$})\\
  r_{ij} &\quad (\mbox{otherwise})
 \end{cases}
\end{align*}
For a set $\Gamma$ of $\{r_{ij}\}_{i \in D_{m}, j\in R_{n}}$-propositional formulae, define
\begin{align*}
\Gamma^{\rho}:=\{\varphi^{\rho}\mid \varphi \in \Gamma\}.
\end{align*}

\end{defn}

\begin{eg}
Let $\rho \in \mathcal{M}^{m}_{n}$ ($m>n$).
 Then, by suitable change of variables, $(injPHP^{m}_{n})^{\rho}$ is equivalent to $injPHP^{m-\#\dom(\rho)}_{n-\#\ran(\rho)}$ (over $AC^{0}$-Frege system, mod $poly(m,n)$-sized proofs).
\end{eg}

\begin{defn}
Let $m>n$.
 Let $T$ be an $injPHP$-tree over $(D_{m},R_{n})$ with height $h$. Given a set $S \subseteq br(T)$ and a family $(T_{b})_{b \in br(T)}$ of $injPHP$-trees where each $T_{b}$ is over $(D_{m}\setminus \dom (b),R_{n} \setminus \ran(b))$, we define \textit{the concatenated tree} 
 \[T*\sum_{b \in S} T_{b}\]
  as follows:
  for each $b \in S$, concatenate $T_{b}$ under $b$ in $T$ identifying the leaf of $b$ and the root of $T_{b}$ (and leaving the label of the root of $T_{b}$).
  
  For two $injPHP$-trees $T$ and $U$ over $(D_{m},R_{n})$,
  we define
  \begin{align*}
  T*U:= T*\sum_{b \in br(T)} U^{b}.
  \end{align*}
\end{defn}

\begin{lemma}\label{numberofbranchproducts}
In the setting of the previous Definition, the following map is a bijection:
\begin{align*}
\{&( b,b'  ) \mid b \in S,\ b' \in br(T_{b}) \}\sqcup (br(T) \setminus S) \rightarrow br\left(T*\sum_{b \in S} T_{b}\right);\\ 
&\begin{cases}
( b,b'  ) &\mapsto bb' \\
b \in br(T) \setminus S &\mapsto b
\end{cases}. 
\end{align*}
 
\end{lemma}

\begin{proof}
Clear.
\end{proof}

\begin{defn}
 Let $\Gamma$ be a subformula closed set of $\{r_{ij}\}_{i \in D_{m}, j\in R_{n}}$-formulae ($m>n$). 
 A \textit{$k$-evaluation (using $injPHP$-trees) of $\Gamma$} is a map 
 \[T_{\cdot} \colon \varphi \in \Gamma \mapsto T_{\varphi}\]
  satisfying the following:
 \begin{enumerate}
  \item Each $T_{\varphi}$ is a labeled $injPHP$-tree over $(D_{m},R_{n})$ with label set $\{0,1\}$ and height $\leq k$.
  \item $T_{0}$ is the $injPHP$-tree with height $0$, whose only branch is labeled by $0$.
  \item $T_{1}$ is the $injPHP$-tree with height $0$, whose only branch is labeled by $1$.
  \item $T_{r_{ij}}$ is the $injPHP$-tree over $(D_{m},R_{n})$ with height $1$, whose label of the root is $i \mapsto ?$ and $br_{1}(T_{r_{ij}})=\{\langle i,j \rangle\}$.
  \item $T_{\lnot \varphi} = T_{\varphi}^{c}$, that is, $T_{\lnot \varphi}$ is obtained from $T_{\varphi}$ by flipping the labels $0$ and $1$.
  \item $T_{\bigvee_{i \in I} \varphi_{i}}$ (where each $\varphi_{i}$ does not begin from $\lor$) represents $\bigcup_{i \in I} br_{1}(T_{\varphi_{i}})$. Here, we say a $\{0,1\}$-labeled $injPHP$-tree $T$ represents a set $\mathcal{F}$ of partial injections if and only if the following hold:
  \begin{enumerate}
   \item For each $b \in br_{1}(T)$, there exists a $\sigma \in \mathcal{F}$ such that $\sigma \subseteq b$.
   \item For each $b \in br_{0}(T)$, every $\sigma \in \mathcal{F}$ satisfies $\sigma \perp b$.
  \end{enumerate}
 \end{enumerate}
\end{defn}

\begin{eg}
Given a list $F=\{ \sigma_{1}, \ldots, \sigma_{N}\}$, where $\sigma_{1}, \ldots, \sigma_{N}$ are partial injections from $D$ to $R$, we define the $\{0,1\}$-labeled $injPHP$-tree $T_{F}$ over $(D,R)$ inductively as follows:
\begin{enumerate}
 \item If $F$ is empty, $T_{F} := T_{0}$.
 \item If some $\sigma_{i}$ is an empty map, $T_{F} := T_{1}$.
 \item Otherwise, ask where to go for each $v \in v(\sigma_{1})$. Let $T$ be the obtained $injPHP$-tree.
 \item For each branch $b \in br(T)$, consider $F^{b}$ below:  
 \begin{align*}
 F^{b}:=\{ \sigma_{i}^{b} \mid \sigma_{i} || b\}.
 \end{align*}
 Construct $T_{F^{b}}$ over $(D\setminus \dom(b), R\setminus \ran(b))$ inductively, and
 set 
 \[T_{F} := T*\sum_{b \in br(T)} T_{F^{b}}.\]
\end{enumerate} 
$T_{F}$ clearly represents $F$ (we regard $F$ as a set here).
\end{eg}

\begin{defn}
Let $T=(S,L)$ be a labeled $injPHP$-tree over $(D_{m},R_{n})$ ($m>n$).
Let $\rho \in \mathcal{M}^{m}_{n}$ and assume $height(T) \leq n- \#\rho$. 
Then \textit{the restricted labeled $injPHP$-tree $T^{\rho}=(S',L')$} is defined as follows:
\[S':=S^{\rho},\ L'(r):=L(r^{-\rho}).\]
\end{defn}

\begin{eg}
If a $\{0,1\}$-labeled $injPHP$-tree $T$ over $(D_{m},R_{n})$ ($m>n$) represents $\mathcal{F}$, $\rho \in \mathcal{M}^{m}_{n}$, and $height(T) \leq n- \#\rho$, then $T^{\rho}$ represents $\mathcal{F}^{\rho}$.
 Indeed, for $r \in br_{1}(T^{\rho})$, there exists $\sigma \in \mathcal{F}$ such that $\sigma \subseteq r^{-\rho}$.
 Hence $\sigma || \rho$ and it gives $\sigma^{\rho} \in \mathcal{F}^{\rho}$, $\sigma^{\rho} \subseteq b^{\rho}$.
 On the other hand, for $r \in br_{0}(T^{\rho})$, each $\sigma \in \mathcal{F}$ satisfies $\sigma \perp r^{-\rho}$. 
 Therefore, for all $\sigma$ such that $\sigma || \rho$, $\sigma^{\rho} \perp (r^{-\rho})^{\rho}=r$ holds.
\end{eg}

\begin{proposition}
Let $T_{\cdot}$ be a $k$-evaluation of a subfomula-closed set $\Gamma$ of $\{r_{ij}\}_{i \in D_{m},j \in R_{n}}$-formulae ($m>n$), $\rho \in \mathcal{M}^{m}_{n}$, $k \leq n-\#\rho$. 

Consider 
\[U_{\varphi}:=(T_{\varphi})^{\rho}.\]
 Note that the RHS is the restricted labeled $injPHP$-tree.
 
 $U_{\varphi}$ is an $injPHP$-tree over $(D_{m} \setminus \dom(\rho), R_{n} \setminus \ran(\rho))$, which can be regarded as $(D_{m-\#\dom(\rho)},R_{n-\#\ran(\rho)})$.
  In particular, we can regard $U_{\cdot}$ as a $k$-evaluation of $\Gamma^{\rho}$ of $\{r_{ij}\}_{i \in D_{m-\#\dom(\rho)},j \in R_{n-\#\ran(\rho)}}$-formulae.
\end{proposition}

\begin{proof}
Clear.
\end{proof}

\begin{theorem}\label{noeval}
Let $f \colon \NN \rightarrow \NN$ be a function satisfying $n<f(n) \leq n^{O(1)}$.
Suppose $(\pi_{n})_{n \geq 1}$ be a sequence of Frege-proofs such that $\pi_{n}$ proves $injPHP^{f(n)}_{n}$ using $UCP^{l,d}_{k}$ as an axiom scheme.

 Then there cannot be a sequence $(T^{n})_{n \geq 1}$ satisfying the following: each $T^{n}$ is an $o(n)$-evaluation of $\Gamma_{n}$ using $injPHP$-trees over $(D_{f(n)},R_{n})$, where $\Gamma_{n}$ is the set of all subformulae appearing in $\pi_{n}$.

\end{theorem}
\begin{proof}
Suppose otherwise. There exist $o(n)$-evaluations $T^{n}$ of $\Gamma_{n}$.
Fix a large enough $n$, and we suppress the superscript $n$ of $T^{n}$, and denote it simply by $T$.

 For $\varphi \in \Gamma_{n}$, define 
 \[T \models \varphi :\Leftrightarrow br_{1}(T_{\varphi})=br(T_{\varphi}).\]
 Then, we can show the following claims analogously with Lemma 15.1.7 and Lemma 15.1.6 in \cite{proofcomplexity}:
\begin{claim}
There exists a constant $c>0$, depending only on the formalization of the Frege system we use, such that the following holds: if $\varphi$ is derived from $\psi_{1}, \ldots, \psi_{k} \in \pi_{n}$ by a Frege rule in $\pi_{n}$, $\forall i \in [k].T \models \psi_{i}$, and $\forall i \in [k]. height(T_{\varphi}) \leq n/c$, then $T \models \varphi$.
\end{claim}
\begin{claim}
 $br_{1}(T_{injPHP^{f(n)}_{n}})= \emptyset$. In particular, $T \not \models injPHP^{f(n)}_{n}$.
\end{claim}
Also, the following fact is useful:
\begin{claim}
\begin{align*}
T \models \bigwedge_{i=1}^{L} \varphi_{i} \Longrightarrow \forall i\in[L].\ T \models \varphi_{i}.
\end{align*}
\end{claim}
\begin{proof}[Proof of the claim]
The hypothesis means 
\[T \models \lnot \bigvee_{i=1}^{L} \lnot \varphi_{i},\] that is, 
\[br_{0}(T_{\bigvee_{i=1}^{L} \lnot \varphi_{i}}) = br(T_{\bigvee_{i=1}^{L} \lnot \varphi_{i}}).\]
Therefore, for each $b \in br(T_{\bigvee_{i=1}^{L} \lnot \varphi_{i}})$ and $r \in br_{1}(T_{\lnot\varphi_{i}})$, $b \perp r$ holds.
Now, assume some $br_{0}(T_{\varphi_{i}})$ is nonempty, and $r_{0}$ be one of its elements.
Since 
\[\#r_{0} \leq o(n) \ \mbox{and} \ height(T_{\bigvee_{i=1}^{L} \lnot \varphi_{i}}) \leq o(n),\] there exists a branch $b \in br(T_{\bigvee_{i=1}^{L} \lnot \varphi_{i}})$ such that $b||r_{0}$, which is a contradiction.
\end{proof}
Therefore, there exists an instance 
\[ I= UCP^{l,d}_{k}[\psi_{i,j,e}/r_{i,j,e}]\]
in $\pi_{n}$ such that $T \not \models I$.
Hence, $k \not \equiv 0 \pmod d$.
 By restricting the formulae and $T$ by some $\rho \in br_{0}(T_{I})$, we obtain the proof $\pi_{n}^{\rho}$ of $(injPHP^{n+1}_{n})^{\rho}$ and the $o(n)$-evaluation $T^{\rho}$ of $\Gamma_{n}^{\rho}$ (note that $\#\rho \leq o(n)$, hence $T^{\rho}$ is well-defined). 
 Therefore, we may assume that $br_{0}(T_{I})=br(T_{I})$.\\
Let
\begin{align*}
S_{i,j,e} &:= T_{\psi_{i,j,e}},\\
S_{e} &:= T_{\bigvee_{(i,j) \in [l] \times [d]}\psi_{i,j,e}},\\
S_{i,j} &:= T_{\bigvee_{e \in [k]} \psi_{i,j,e}},\\
P_{i} &:= T_{\bigvee_{j \in [d]} \lnot \bigvee_{e \in [k]}\psi_{i,j,e}},\\
N_{i} &:= T_{\bigvee_{j \in [d]} \bigvee_{e \in [k]}\psi_{i,j,e}},\\
U_{i} &:= T_{(\lnot\bigvee_{j \in [d]} \lnot \bigvee_{e \in [k]}\psi_{i,j,e}) \lor (\lnot \bigvee_{j \in [d]} \bigvee_{e \in [k]}\psi_{i,j,e})}.\\
&(i \in [l], j \in [d], e \in [k])
\end{align*}

Since $br_{0}(T_{I})=br(T_{I})$, we observe the following facts:
\begin{enumerate}
 \item\label{etot} For each $e \in [k]$, $T \models \bigvee_{(i,j) \in [l] \times [d]}\psi_{i,j,e}$, that is, every $b \in br(S_{e})$ has $(i,j) \in [l] \times [d]$ and $b' \in br_{1}(S_{i,j,e})$ such that $b' \subseteq b$.
 \item\label{einj} For each $e \in [k]$ and $(i,j) \neq (i^{\prime}, j^{\prime}) \in [l] \times [d]$, every pair of branches $b \in br_{1}(S_{i,j,e})$ and $b^{\prime} \in br_{1}(S_{i^{\prime},j^{\prime},e})$ satisfies $b \perp b^{\prime}$.
 \item\label{ijinj} For each $e \neq e^{\prime} \in [k]$ and $(i,j) \in [l] \times [d]$, each $b \in br_{1}(S_{i,j,e})$ and $b^{\prime} \in br_{1}(S_{i,j,e^{\prime}})$ satisfies $b \perp b^{\prime}$.
 \item\label{iquasitot} For each $i \in [l]$, $T \models (\lnot\bigvee_{j \in [d]} \lnot \bigvee_{e \in [k]}\psi_{i,j,e}) \lor (\lnot \bigvee_{j \in [d]} \bigvee_{e \in [k]}\psi_{i,j,e})$, that is, every $b \in br(U_{i})$ is an extension of some $b^{\prime} \in br_{1}(P_{i}^{c})$ or some $b^{\prime} \in br_{1}(N_{i}^{c})$. In the former case, $b$ is incompatible with every $b^{\prime\prime}\in \bigcup_{j}br_{0}(S_{i,j})$. In the latter case, $b$ is incompatible with every $b^{\prime\prime} \in \bigcup_{j,e}br_{1}(S_{i,j,e})$. Therefore, the two cases are mutually disjoint. Indeed, if $b$ satisfies the both cases, then take $b^{\prime} \in br(S_{i,j})$ such that $b || b^{\prime}$ (which exists since $\#b$ and $height(S_{i,j})$ are both $o(n)$). It follows that $b^{\prime} \in br_{1}(S_{i,j})$, and therefore $b^{\prime}$ is an extension of some $b^{\prime \prime } \in \bigcup_{e}br_{1}(S_{i,j,e})$, which contradicts the observation of the latter case. 
\end{enumerate}
With the observations above, we construct labeled $injPHP$-trees $(X_{i,j})_{(i,j) \in [l] \times [d]}$ and $(Y_{e})_{e \in [k]}$ as follows：
\begin{itemize}
 \item We define $Y_{e}$ for fixed $e$ first. 
 Consider $S_{e}$. 
 By observation \ref{etot} and \ref{einj}, each $b \in br(S_{e})$ has a unique $(i_{b},j_{b}) \in [l] \times [d]$ such that $b$ is an extension of some $b^{\prime} \in br_{1}(S_{i_{b},j_{b},e})$. 
 Consider the tree 
 \begin{align*}
 S_{e}*\sum_{b \in br(S_{e})}(U_{i_{b}}*S_{i_{b},j_{b}})^{b}
 \end{align*}
 (here, we have concatenated the trees, ignoring their labels).\\
 \quad Label each branch extending $b \in br(S_{e})$ with $\langle i_{b}, j_{b},e\rangle$. 
 Let $Y_{e}$ be the resulting labeled $injPHP$-tree. 
 Note that $height(Y_{e})$ is still $o(n)$.
 \item Next, we define $X_{i,j}$ for fixed $(i,j) \in [l]\times[d]$. 
 Consider $U_{i}$. 
 Let $B \subseteq br(U_{i})$ be the set of all $b \in br(U_{i})$ satisfying the former case of observation \ref{iquasitot}. 
 Consider the tree $\widetilde{X}_{i,j} := U_{i}*\sum_{b \in B}S_{i,j}^{b}$. 
 Each branch $\widetilde{b} \in br(\widetilde{X}_{i,j})$ satisfies one of the following:
 \begin{enumerate}
  \item $\widetilde{b} \in br(U_{i}) \setminus B$.
  \item Otherwise, by Lemma \ref{numberofbranches}, we can decompose $\widetilde{b}=bs^{b}$ ($b \in B, s \in br(S_{i,j})$).
  Since $b$ is an extension of some $b^{\prime} \in br_{1}(P_{i}^{c})$, we have $br_{1}(S_{i,j}^{b})=br(S_{i,j}^{b})$.
  It implies $s \in br_{1}(S_{i,j})$.
  Therefore, by observation \ref{ijinj}, each $s^{b} \in br(S_{i,j}^{b})$ has a unique $e_{\widetilde{b}} \in [k]$ such that $s$ extends some $b^{\prime\prime} \in br_{1}(S_{i,j,e_{\widetilde{b}}})$.
 \end{enumerate}
 Let $\widetilde{B}$ be the all branches of $\widetilde{X}_{i,j}$ satisfying the second item.
 We define
 \begin{align*}
 \widehat{X}_{i,j} := \widetilde{X}_{i,j}*\sum_{\widetilde{b} \in \widetilde{B}} (S_{e_{\widetilde{b}}})^{\widetilde{b}}.
 \end{align*}
 We label each branch $b \in br(\widehat{X}_{i,j})$ as follows and define $X_{i,j}$ to be the obtained labeled $injPHP$-tree.
 \begin{enumerate}
  \item If $b \in br(U_{i}) \setminus B$, then label it with the symbol $\perp$.
  \item Otherwise, there exists a unique $\widetilde{b} \in \widetilde{B}$ such that $\widetilde{b} \subseteq b$.
  Label the branch $b$ with $\langle i, j, e_{\widetilde{b}} \rangle$.
 \end{enumerate}
\end{itemize}

By observation \ref{ijinj}, we see that $(X_{i,j})_{i \in [l],j \in [d]}$ and $(Y_{e})_{e \in [k]}$ satisfy the following:
\begin{itemize}
 \item For each $i$，$br_{\perp}(X_{i,1})= \cdots = br_{\perp}(X_{i,d})$.
  \item For each $i,j,e$, $br_{\langle i,j,e \rangle}(X_{i,j})= br_{\langle i,j,e \rangle}(Y_{e})$ (as sets of partial injections). 
\end{itemize}
The second item is justified as follows.
By Lemma \ref{numberofbranches}, Lemma \ref{numberofbranchproducts}, and the definitions of $Y_{e}$ and $X_{i,j}$, the following  are bijections for each $i,j,e$:
\begin{align*}
&\Bigg\{ \left( b,\beta,s \right) \in br(S_{e}) \times br(U_{i}) \times br(S_{ij}) \mid \exists b'\in br_{1}(S_{i,j,e}).(b' \subseteq b),\ \beta || s,\ b||\beta s \Bigg\} \\
&\rightarrow br_{\langle i,j,e \rangle}(Y_{e});\ ( b,\beta,s ) \mapsto b(\beta s).
\end{align*}

\begin{align*}
&\Bigg\{ \left( \beta,s,b \right) \in br(U_{i}) \times br(S_{i,j}) \times br(S_{e}) \mid 
 \begin{array}{l}
\exists \hat{b}\in br_{1}(P_{i}^{c}).(\hat{b} \subseteq \beta),\\
 \exists b'' \in br_{1}(S_{i,j,e}). (b'' \subseteq s),\\
 \beta || s ,\ \beta s || b 
\end{array}
\Bigg\}\\
&\rightarrow br_{\langle i,j,e \rangle}(X_{i,j});\ (\beta,s,b ) \mapsto (\beta s)b.
\end{align*}
Hence, in order to see $br_{\langle i,j,e \rangle}(X_{i,j})=br_{\langle i,j,e \rangle}(Y_{e})$, it suffices to show the following:
let $\beta \in br(U_{i})$, $s \in br(S_{i,j})$, $b \in br(S_{e})$.
If $\beta || s$ and $\beta s || b$, then the following are equivalent:
\begin{enumerate}
 \item\label{Yepart} There exists $b'\in br_{1}(S_{i,j,e})$ such that $b' \subseteq b$.
 \item\label{Xijpart} $\beta$ is an extension of some $\hat{b}\in br_{1}(P_{i}^{c})$, and $s$ is an extension of some $b'' \in br_{1}(S_{i,j,e})$.
\end{enumerate}

We consider (\ref{Yepart}) $\Rightarrow$ (\ref{Xijpart}) first.
Suppose $\beta$ is not an extension of any $\hat{b}\in br_{1}(P_{i}^{c})$.
$\beta \in br(U_{i})$, so it must be an extension of some $r \in br_{1}(N_{i}^{c})=br_{0}(N_{i})$ by observation \ref{iquasitot} above. 
Thus we have $r \perp b'$, and therefore $\beta \perp b$, contradicting the assumption $\beta s || b$.
Hence, $\beta$ is an extension of some $\hat{b}\in br_{1}(P_{i}^{c})=br_{0}(P_{i})$.
Then, by observation \ref{iquasitot} and the assumption $\beta || s$, we have $s \in br_{1}(S_{i,j})$, which means $s$ is an extension of some $b'' \in br_{1}(S_{i,j,e})$.
This finishes the proof of  (\ref{Yepart}) $\Rightarrow$ (\ref{Xijpart}).

Next, we consider the converse.
By observation \ref{etot}, there exist $i',j'$ and $b' \in br(S_{i',j',e})$ such that $b' \subseteq b$.
Then we have $b' || b''$ since $b||s$ follows from $\beta s || b$.
By observation \ref{einj}, $(i,j)=(i',j')$ and $b'=b''$ follows.
Since $b'' \in br_{1}(S_{i,j,e})$ by assumption, we have the converse.

Now, we construct a low-degree $NS$-refutation of $\lnot inj^{*}PHP^{f(n)-\#\dom(\rho)}_{n-\# \ran(\rho)}$.
 By the properties of $(X_{i,j})_{i \in [l], j \in [d]}$ and $(Y_{e})_{e \in [k]}$, the following holds:
\begin{align*}
 \sum_{(i,j) \in [l]\times [d]}\sum_{\alpha \in br(X_{i,j})} x_{\alpha} = \sum_{e \in [k]} \sum_{\beta \in br(Y_{e})} x_{\beta} \pmod d.
\end{align*}
(Here, for a partial injection $\rho$, $x_{\rho} := (\prod_{\langle p,h\rangle  \in \rho_{bij}}x_{p,h})( \prod_{\langle h \rangle \in \rho_{sing}}u_{h})$).

 It is easy to see that for any $injPHP$-tree $T$ over $(D_{M},R_{m})$, 
$\sum_{\alpha \in br(T)}x_{\alpha}=1$ has a $NS$-refutation from $\lnot inj^{*}PHP^{M}_{m}$ with degree $\leq height(T)$.
Hence, we obtain a $NS$-proof of $k=0$ from $\lnot inj^{*}PHP^{f(n)-\#\dom(\rho)}_{n-\#\ran(\rho)}$ over $\ZZ/d\ZZ$ with degree $o(n)$.
Since $k \not \equiv 0 \pmod d$, it is a $NS$-refutation of $\lnot inj^{*}PHP^{f(n)-\#\dom(\rho)}_{n-\#\ran(\rho)}$ over $\ZZ/d\ZZ$ with degree $o(n)$.
However, it contradicts Proposition \ref{injPHPandinj*PHPareequivalent} and Corollary \ref{Razborovlowerboundforfinitering}.

\end{proof}

Setting $f(n):=n+1$ and taking the contrapositive of Theorem \ref{noeval}, we obtain that an analogue of switching lemma for $injPHP$-trees suffices to prove Conjecture \ref{UCPinjPHP}:
\begin{cor}
Assume $F_{c} + UCP^{l,d}_{k} \vdash_{poly(n)} injPHP^{n+1}_{n}$ is witnessed by $AC^{0}$-proofs $(\pi_{n})_{n \geq 1}$.
Suppose there are partial injections $(\rho_{n})_{n \geq 1}$ satisfying
\begin{itemize}
 \item For each $n$, $\rho_{n} \in \mathcal{M}^{n+1}_{n}$.
 \item $n-\#\ran(\rho_{n}) \rightarrow \infty$ $(n \rightarrow \infty)$.
 \item There exist $o(n-\#\ran(\rho_{n}))$-evaluations $(T^{n})_{n\geq 1}$ of $\Gamma_{n}^{\rho}$, where $\Gamma_{n}$ is the all subformulae appearing in $\pi_{n}$.
\end{itemize}
Then, we obtain a contradiction.
\end{cor}

Note that an analog of switching lemma would give how to construct such $\rho_{n}$ above.

\begin{rmk}\label{switchingforinjPHPisdifficult}
It will be good if we can prove such an analog for $injPHP$-trees; in that case, we can apply the previous theorem and prove that Conjecture \ref{UCPdoesnotproveinjPHP} is valid.
However, it seems implementing this approach is beyond the current proof techniques.
The difficulty is relevant to that of the famous open problem: does $V^{0} \vdash injPHP^{2n}_{n}$ hold?
\end{rmk}

\section{On the strength of the oddtown theorem}\label{On the strength of the oddtown theorem}
\quad The oddtown theorem is a combinatorial principle stating that there cannot be $n+1$ orthogonal normal vectors in $\FF_{2}^{n}$. In other words, (regarding each $v \in \FF^{n}_{2}$ as the characteristic vector of a subset $S \subseteq [n]$) there cannot be a family $(S_{i})_{i \in [n+1]}$
satisfying the following:
\begin{itemize}
 \item Each $S_{i}$ has an odd cardinality.
 \item Each $S_{i} \cap S_{i^{\prime}}$ ($i < i^{\prime}$) has an even cardinality.
\end{itemize}
Historically, the oddtown theorem and Fisher's inequality (introduced in \S \ref{On the strength of Fisher's inequality}) were proposed as candidates for statements that are easy to prove in the extended Frege system but not in the Frege system (\cite{Hard}).
Standard proofs of them use linear-algebra methods, such as evaluating the number of linearly independent vectors of a given linear space, which do not seem easy to formalize in the bounded arithmetic $VNC^{1}$ (corresponding to polynomial-sized Frege proofs), and this was why they and some other combinatorial principles were presented as above candidates.

Actually, around the time \cite{Hard} was published, a linear-algebra-free proof of Fisher's inequality was already discovered, and I did not know it.
See \S \ref{On the strength of Fisher's inequality} for details and acknowledgment.

On the other hand, upper bounds of linear algebra methods were also given in a series of works.
\cite{Soltys} and \cite{The proof complexity of linear algebra} gave natural bounded arithmetics $\LA$ and $\LAP$ capable of formalizing elementary arguments in linear algebra, formalized the matrix determinant by Berkowitz algorithm \cite{Berkowitz}, and showed the following:
\begin{itemize}
 \item Cayley-Hamilton, co-factor expansion, and an axiomatic definition of determinants (multi-linear, alternating, and $\det(I)=1$ for unit matrices $I$) are equivalent over $\LAP$.
 \item  Multiplicativity of determinants implies all above in $\LAP$.
\end{itemize} 
\cite{Short proofs for the determinant identities} and \cite{Uniform} were recent breakthroughs in the area.

  \cite{Short proofs for the determinant identities} gave a short proof of multiplicativity of determinants of $\FF$-matrices in the proof system $P_{c}(\FF)$, which manipulates arithmetical circuits, for general coefficient fields $\FF$.
  The results and the arguments in \cite{Short proofs for the determinant identities} also imply multiplicativity of determinant for $\ZZ$-matrices (or, equivalently, $\QQ$-rational matrices) has quasipolynomial-sized Frege proofs.
  See the introduction of \cite{Short proofs for the determinant identities} for details.
  
   \cite{Uniform} carefully formalized the argument for $\ZZ$-matrices in $\VNC^{2}$.
Using the result of \cite{Uniform}, we can show that $\VNC^{2}$ proves the oddtown theorem, and therefore propositional translations have quasipolynomial-sized Frege proofs.
See Proposition \ref{VNC^{2}provesoddtown}.

However, there is still room to investigate the precise strengths of the oddtown theorem and Fisher's inequality in terms of bounded reverse mathematics.
In this section, we focus on lower bounds for the oddtown theorem.
We observe that a natural formalization of the oddtown theorem over $V^{0}$ is stronger than several combinatorial principles related to counting.
Then, we try to analyze the relation between the oddtown theorem and $Count^{p}_{n}$ for $p$, which is not a $2$-power.

\begin{defn}
Define the $\Sigma^{B}_{0}$ $\mathcal{L}^{2}_{A}$-formula $oddtown(n,P,Q,R,S)$ as follows:
\begin{align*}
\lnot 
&[\forall i \in [n+1]. \forall j \in [n]. (S(i,j) \leftrightarrow Q(i,j) \lor \exists e \in [n]^{(2)}. (j \in^{*} e \land P(i,e)) \\
&\land \forall i \in [n+1]. \exists j \in [n]. Q(i,j)\\
&\land \forall i \in [n+1]. \forall j \neq j^{\prime}\in [n].(\lnot Q(i,j) \lor \lnot Q(i,j^{\prime})) \\
&\land \forall i \in [n+1]. \forall j\in [n]. \forall e \in [n]^{(2)}(j \in^{*} e \rightarrow \lnot Q(i,j) \lor \lnot P(i,e)) \\
&\land \forall i \in [n+1]. \forall e \neq e^{\prime} \in [n]^{(2)}(e \cap e^{\prime} \neq \emptyset \rightarrow \lnot P(i,e) \lor \lnot P(i,e^{\prime})) \\
&\land 
\forall i < i^{\prime} \in [n+1]. \forall j \in [n]. (S(i,j) \land S(i^{\prime},j) \leftrightarrow \exists e \in [n]^{(2)} (j \in^{*} e \land R(i,i^{\prime},e)) )\\
&\land \forall i < i^{\prime} \in [n+1]. \forall e\neq e^{\prime} \in [n]^{(2)}. (e \cap e^{\prime} \neq \emptyset \rightarrow \lnot R(i,i^{\prime},e) \lor \lnot R(i,i^{\prime},e^{\prime}))]
\end{align*}
\end{defn}
Intuitively, $S$ above gives sets $S_{i}:=\{ j \in [n] \mid S(i,j)\}$, $P$ gives a $2$-partition of each $S_{i}$ leaving one element, which is specified by $Q$, 
and $R$ gives a $2$-partition of each $S_{i} \cap S_{i^{\prime}}$ ($i < i^{\prime}$).
\begin{defn}
Define the propositional formula $oddtown_{n}$ as follows:
\begin{align*}
&oddtown_{n}:=
\begin{cases}
&1 \quad (n=0)\\
\lnot[ 
&\bigwedge_{i \in [n+1]} \bigwedge_{j \in [n]} (\lnot s_{ij} \lor q_{ij} \lor \bigvee_{e: j \in e \in [n]^{(2)}} p_{ie}) \\
& \land \bigwedge_{i \in [n+1]} \bigwedge_{j \in [n]} (s_{ij} \lor \lnot q_{ij})\\
& \land \bigwedge_{i \in [n+1]} \bigwedge_{j \in [n]}\bigwedge_{e: j \in e \in [n]^{(2)}} (s_{ij} \lor \lnot p_{ie})\\
& \land \bigwedge_{i \in [n+1]} \bigvee_{j \in [n]} q_{ij}\\
&\land \bigwedge_{i \in [n+1]} \bigwedge_{j < j^{\prime} \in [n]} (\lnot q_{ij} \lor \lnot q_{ij^{\prime}}) \\
& \land \bigwedge_{i \in [n+1]} \bigwedge_{j \in [n]} \bigwedge_{e:j \in e \in [n]^{(2)}} (\lnot q_{ij} \lor \lnot p_{ie})\\
&\land \bigwedge_{i \in [n+1]} \bigwedge_{e\perp e^{\prime} \in [n]^{(2)}} (\lnot p_{ie} \lor \lnot p_{ie^{\prime}})\\
& \land \bigwedge_{i<i^{\prime} \in [n+1]} \bigwedge_{j \in [n]} (\lnot s_{ij} \lor \lnot s_{i^{\prime}j} \lor \bigvee_{e: j \in e \in [n]^{(2)}} r_{ii^{\prime}e}) \\
& \land \bigwedge_{i<i^{\prime} \in [n+1]} \bigwedge_{j \in [n]}\bigwedge_{e: j \in e \in [n]^{(2)}}  ( s_{ij} \lor \lnot r_{ii^{\prime}e}) \\
& \land \bigwedge_{i<i^{\prime} \in [n+1]} \bigwedge_{j \in [n]}\bigwedge_{e: j \in e \in [n]^{(2)}}  ( s_{i^{\prime}j} \lor \lnot r_{ii^{\prime}e}) \\
& \land \bigwedge_{i<i^{\prime} \in [n+1]} \bigwedge_{e \perp e^{\prime} \in [n]^{(2)}} (\lnot r_{ii^{\prime}e} \lor \lnot r_{ii^{\prime}e^{\prime}}) ] \quad (n \geq 1)
\end{cases}
\end{align*}
\end{defn}

By a reason similar to that of Convention \ref{excuse}, we abuse the notation and write $oddtown_{n}$ to express $oddtown(n,P,Q,R,S)$, too.

 We first observe an upper bound of the strength of $oddtown_{n}$:
 \begin{proposition}\label{VNC^{2}provesoddtown}[A corollary of \cite{Uniform}]
 \[\VNC^{2} \vdash oddtown(n,P,Q,R,S).\]
 \end{proposition}
 For the definition of $\VNC^{2}$, see \cite{Cook}.
 
 \begin{proof}[Proof Sketch]
  \cite{Uniform} showed that there exists a $\Sigma^{B}_{1}$-definition $\det(A)=N$ of the matrix determinant, where both $A
  $ and $N$ are of set-sort, 
  $N$ is meant to be a binary representation of an integer,
  and $A$ is meant to code a square matrix with integer coefficients, where each coefficient is represented by a binary expression, such that $\VNC^{2}$ proves:
  \begin{itemize}
  \item Let $A$ and $B$ be strings coding square matrices with integer coefficients of the same size $(k \times k)$.
   Then $\det(A)\det(B)=\det(AB)$.
   \item $\det(A)$ can be computed by co-factor expansion concerning any row or column. 
  \end{itemize}
  
  Focusing on the last digit of the integers involved, we can see that there exists a $\Sigma^{B}_{1}$-definition $\det_{\FF_{2}}(A)=b$, where $b \in \{0,1\}$, and $A$ is of set-sort, meant to code a square matrix with $\FF_{2}$-coefficients satisfying the above two items.
  
  Armed with this, we work in $\VNC^{2}$ and show $oddtown_{n}$ as follows.
  Towards a contradiction, suppose $n,P,Q,R,S$ violate $oddtown(n,P,Q,R,S)$.
  Let $M$ be an $\FF_{2}$ matrix of size $(n+1)\times n$ given by 
  \[M_{ij} := \begin{cases}
  1 \quad &(\mbox{if $S(i,j)$ holds})\\
  0 \quad &(\mbox{otherwise})
  \end{cases}.\]
  Since $oddtown(n,P,Q,R,S)$ is violated, we have $MM^{t}=I_{n+1}$ as $\FF_{2}$-matrices, where $I_{n+1}$ is a unit matrix of size $(n+1)\times (n+1)$.
  Note that $\VNC^{2}$ is an extension of $\VTC^{0}$ and can count cardinalities of bounded sets, which is sufficient to prove $MM^{t}=I_{n+1}$.
  Consider the matrix $A$ obtained by padding $M$ by a zero column vector.
  $A$ is of size $(n+1)\times (n+1)$, and $AA^{t}=I_{n+1}$.
  Taking determinants of both sides and applying multiplicativity of determinants, we have $\det_{\FF_{2}}(A)\det_{\FF_{2}}(A^{t})=\det_{\FF_{2}}I_{n+1}$.
  However, $A$ has a zero column vector, and $A^{t}$ has a zero row vector.
  Therefore, by co-factor expansions, we have $\det_{\FF_{2}}(A)=\det_{\FF_{2}}(A^{t})=0$.
  On the other hand, it is elementary to see $\det_{\FF_{2}}(I_{n+1})=1$, a contradiction.
 \end{proof}
 
 \begin{cor}
 There exist $2^{O((\log n)^{2})}$-sized Frege proofs of $oddtown_{n}$.
 \end{cor}
 
Now, we focus on the lower bounds of the strength of $oddtown_{n}$.
 We start with the following simple observations: 
 \begin{proposition}\label{Propoddtown}
\begin{enumerate}
\item\label{oddtowninjPHP} $V^{0} +oddtown_{k} \vdash injPHP^{n+1}_{n}$.
\item\label{oddtownCount2} $V^{0}+oddtown_{k} \vdash Count^{2}_{n}$.
\end{enumerate}
\end{proposition}
\begin{proof}
We first prove \ref{oddtowninjPHP}. Argue in $V^{0}$. We prove the contrapositive. Suppose there exists an injection $f \colon [n+1] \rightarrow [n]$. 
Define 
\[S_{i}:=\{f(i)\}.\] 
Then it violates $oddtown_{n}$.

 We next prove \ref{oddtownCount2}. Argue in $V^{0}$. 
We prove the contrapositive. Suppose $[2n+1]$ is partitioned by 2-sets. Let $R$ be the 2-partition. 
Then setting 
\[S_{i}:=[2n+1] \quad (i \in [2n+2]),\]
 we can violate $oddtown_{2n+2}$ since $[2n+1] \cap [2n+1]$ can be 2-partitioned by $R$ while $[2n+1]\setminus\{2n+1\}$ has a natural $2$-partition.
\end{proof}

\begin{rmk}
Observing the proof above, one may think that we might obtain another interesting formalization of the oddtown theorem imposing $\{S_{i}\}_{i \in [n+1]}$ to be a family of \textit{distinct} sets. Let $oddtown^{\prime}$ be this version.
It turns out that:
\begin{enumerate}
 \item\label{oddtownoddtownprime} $V^{0}+oddtown_{k} \vdash oddtown_{n}^{\prime}$.
 \item\label{oddtownprimeCount2oddtown} $V^{0}+oddtown_{k}^{\prime}+Count^{2}_{k} \vdash oddtown_{n}$.
 \item\label{oddtownprimeCount2} $V^{0}+oddtown_{k}^{\prime} \vdash Count^{2}_{n}$.
\end{enumerate}
Hence, $oddtown_{n}$ and $oddtown^{\prime}_{n}$ have the same strength over $V^{0}$.
\end{rmk}
\begin{proof}[The proof of the remark]
The item \ref{oddtownoddtownprime} is clear.

We show the item \ref{oddtownprimeCount2oddtown}. Work in $V^{0}+oddtown_{k}^{\prime}+Count^{2}_{k}$. 
Assume $\{S_{i}\}_{i=1}^{n+1}$ (where $S_{i} \subseteq [n]$) violates $oddtown_{n}$.
Since each $S_{i} \cap S_{i^{\prime}}$ ($i < i^{\prime}$) is $2$-partitioned, it follows that $S_{i} \neq S_{i^{\prime}}$.
Indeed, if $S_{i}=S_{i^{\prime}}=:S$, then both $S$ and $S\setminus \{s_{0}\}$ ($s_{0} \in S$) is 2-partitioned by the hypothesis.
Consider a straightforward bijection :
\begin{align*}
[2n-1] \cong ([n]\setminus S) \sqcup ([n]\setminus S) \sqcup S \sqcup (S \setminus \{s_{0}\}).
\end{align*}
The right-hand side gives a natural 2-partition using those of $S$ and $S \setminus \{s_{0}\}$, and it induces a $2$-partition of the left-hand side, which violates $Count^{2}_{2n-1}$.\\
\quad Hence, it follows that each $S_{i}$ is distinct.
However, it contradicts $oddtown_{n}^{\prime}$. 
This shows the item \ref{oddtownprimeCount2oddtown}.

Lastly, we show the item \ref{oddtownprimeCount2}.
 Argue in $V^{0}$. 
 Assume $R$ is a $2$-partition of $[2n+1]$.
  Note that $R$ has a natural linear ordering induced by that of whole numbers. 
Define $\{S_{i}\}_{i=1}^{2n+2}$ as follows: it is easy to see that $R$ has at least four elements. Take distinct $r_{1}, r_{2}, r_{3} \in R$. For each $i = 1, \ldots, 2n+1$, take the unique $j$ such that $\{i,j\} \in R$. Then, 
\begin{enumerate} 
\item[Case1.] If $i<j$, set $S_{i}:=[2n+1] \setminus \{i,j\}$.
\item[Case2.] If $i>j$, set $S_{i}:=[2n+1] \setminus (\{i,j\} \cup s_{i})$, where $s_{i}$ is the successor of $\{i,j\}$ in $R$. If there is none (i.e. $\{i,j\} =\max R$), let $s_{i}$ be $\min R$.
\end{enumerate}
Furthermore, we define $S_{2n+2}:=[2n+1] \setminus (r_{1} \cup r_{2} \cup r_{3})$.

Now, we see that $\{S_{i}\}_{i=1}^{2n+2}$ violates $oddtown_{2n+2}$. Indeed, the $S_{i}$ are distinct, we can take natural 2-partitions leaving one element for each $S_{i}$, and we can obtain 2-partitions for each $S_{i} \cap S_{j}$ ($i<j$) removing at most five elements from $R$.
\end{proof}

By theorem \ref{CountpinjPHP} and \ref{injPHPCountp}, we obtain

\begin{cor}
\begin{align*}
 &V^{0} + injPHP^{k+1}_{k} \not \vdash oddtown_{n}, \\
 &V^{0} + Count^{2}_{k} \not \vdash oddtown_{n}.
\end{align*}
\end{cor}

This raises the following natural problems:
\begin{question}
\begin{enumerate}
 \item $V^{0}+ injPHP^{k+1}_{k} + Count^{2}_{k} \vdash oddtown_{n}$? How about $V^{0}+ GCP \vdash oddtown_{n}$?
 \item\label{oddtownCountp} $V^{0}+oddtown_{k} \vdash Count^{p}_{n}$ for which $p$?
\end{enumerate}
\end{question}
The author cannot answer these questions for now. 
However, we tackle the item \ref{oddtownCountp} in the rest of this section.

By Proposition \ref{Propoddtown} and Theorem \ref{charcount}, it is easy to see:
\begin{cor}\label{oddtownCount2power}
 If $p$ is a power of $2$, $V^{0} + oddtown_{k} \vdash Count^{p}_{n}$.
\end{cor}
The author conjectures that the converse of this corollary holds. Furthermore, the author conjectures the following:

\begin{conj}\label{conjoddtownCount}
For each $d \in \NN$ and a prime $p \neq 2$, 
\[F_{d}+oddtown_{k} \not \vdash_{poly(n)} Count^{p}_{n}.\]
\end{conj}

Using Theorem \ref{charcount}, it is easy to see that Conjecture \ref{conjoddtownCount} implies the converse of Corollary \ref{oddtownCount2power}.
We give a sufficient condition to prove Conjecture \ref{conjoddtownCount}:

\begin{theorem}\label{OnF2}
Let $p \in \NN$ be a prime other than $2$. 
  Suppose 
  \[F_{d}+oddtown_{k} \vdash_{poly(n)} Count^{p}_{n}.\] 
  Then, for any large enough $n \not \equiv 0 \pmod p$, there exist $m =n^{O(1)}$ and a family $(f_{ij})_{i \in [m+1], j \in [m]}$ of polynomials over $\FF_{2}$ such that the following equalities have $NS$-proofs over $\FF_{2}$ from $\lnot Count^{p}_{n}$ with degree $O(1)$:
  \begin{align*}
   &\sum_{j \in [m]}f_{ij} =1 \quad & (i \in [m+1])\\
   &\sum_{j \in [m]} f_{ij}f_{i^{\prime}j}=0 \quad &(i \neq i^{\prime} \in [m+1])
  \end{align*}

Here, $\lnot Count^{p}_{n}$ (where $n \not \equiv 0 \pmod p$) means the following system of polynomials:
 \begin{align*}
  &\sum_{e : j \in e \in [n]^{(p)}} x_{e}-1,\ x_{e}x_{e^{\prime}},\ x_{e}^{2}-x_{e}\\
   &(j \in [n] \ \& \ e, e^{\prime} \in [n]^{(p)} \ \& \ e \perp e^{\prime})
 \end{align*}
\end{theorem}
Hence, if we can prove a non-constant degree lower bound for $NS$-proofs for the above system of equations, then Conjecture \ref{conjoddtownCount} is true.

Before the proof, we present a strategy for obtaining the constant degree bound above.
We assume the basics of \textit{partial $p$-partitions}, \textit{$p$-trees}, and \textit{$k$-evaluations using $p$-trees}.
See section 15.5 in \cite{proofcomplexity} for references.
We also use the notations $br(T)$, $T \models \varphi$, $T*\sum_{b \in S} T_{b}$, etc. as the straightforward analogous meanings to the ones given in \S \ref{On Question UCPinjPHP}.

The power of the notion comes from the following so-called switching lemma:
\begin{theorem}[Lemma 12.3.11 in \cite{Krajicek}. Essentially by \cite{remake}, \cite{remake2}.]\label{switchinglemmaforptrees}
Let $p \geq 2$.
Then there exists a constant $c$ satisfying the following:
let $H_{1},\ldots, H_{N}$ be families of partial $p$-partitions of $[pn+r]$, where $0 \leq r<p$.
Assume that $\#H_{i} \leq t \leq s$ for every $i \in [N]$, and 
\begin{align}
\frac{w^{s}}{[t\cdot (n-w)]^{c \cdot s}} > N, \label{assumptionforswitching}
\end{align}
where $w \leq n$.
Then there is a partial $p$-partition $\rho$ with $\# \rho = w$, such that for every $i \in [N]$ there exists a $p$-tree with height at most $p\cdot s$ representing $H_{i}^{\rho}$.
\end{theorem}

\begin{rmk}
In the book \cite{Krajicek}, $r=1$ is assumed, but the proof is straightforwardly extended to treat general $r < p$ since $p$ is a fixed constant here.
\end{rmk}

\begin{rmk}
 In the book \cite{Krajicek}, the notion \textit{$k$-complete system} is used instead of $p$-trees, and therefore the precise definition of $k$-evaluation in it is different from the one given in \cite{proofcomplexity}.
 However, the proof of Lemma 12.3.11 in \cite{Krajicek} essentially constructs canonical $p$-trees $T_{i}$ representing $H_{i}$'s under ``good'' restrictions, and $br(T_{i})$'s are the $k$-complete systems argued in \cite{Krajicek}.
 In other words, an inequality corresponding to (\ref{assumptionforswitching}) can be extracted from the proof of Theorem 15.2.2 in \cite{proofcomplexity}.
 The intertwin between several terminologies and notions is concisely mentioned in \S 4.4 of \cite{partiallydefinableforcing}.
\end{rmk}

In particular, when $N=n^{O(1)}$, choosing a parameter $0<\epsilon<1/c$ and putting $w:=n-n^{\epsilon}$, we can satisfy the inequality (\ref{assumptionforswitching}) by choosing sufficiently large $t=s=O(1)$.
Applying Theorem \ref{switchinglemmaforptrees} constantly many times, we obtain the following:

\begin{cor}\label{constructionofO(1)-evaluation}[A counterpart of Theorem 12.4.3 in \cite{Krajicek} when $\# \Gamma$ is just $n^{O(1)}$]
Let $e,d \geq 1$. 
Then, for sufficiently small $\epsilon >0$, there exists a constant $k>0$ satisfying the following: 
\begin{itemize}
 \item For sufficiently large $n$, if $\Gamma$ is a set of formulae of depth $\leq d$ whose variables are those used in $Count^{p}_{pn+r}$ ($0 \leq r < p$), closed under the subformulae, and with size $\#\Gamma \leq n^{e}$, then there is a partial $p$-partition $\rho$ with size $n-n^{\epsilon}$ such that there exists a $k$-evaluation of $\Gamma^{\rho}$.  
 \end{itemize}
 Here, $n^{\epsilon}$ is rounded to an integer in an appropriate way.

\end{cor}

Now, we move on to the proof of Theorem \ref{OnF2}.
\begin{proof}[Proof of Theorem \ref{OnF2}]
In this proof, we assume $p=3$ for readability.
Let proofs $(\pi_{l})_{l \in \NN}$ witness 
\begin{align*}
F_{d}+oddtown_{k} \vdash_{poly(l)} Count^{3}_{l}.
\end{align*}
We focus on $l \not\equiv 0 \pmod 3$.
Let $l=3n+r$ ($0 < r < 3$).
Let $\Gamma_{l}$ be the all subformulae appearing in $\pi_{l}$.
Apply Corollary \ref{constructionofO(1)-evaluation} for $3$-trees, and obtain $\epsilon > 0$ and a partial $3$-partition $\rho$ with size $n^{\epsilon}$ such that there exists an $O(1)$-evaluation $T$ of $\Gamma_{l}^{\rho}$, where $l$ is sufficiently large.
Note that $(Count^{3}_{l})^{\rho}$ can be identified as $Count^{3}_{3n^{\epsilon}+r}$.

It suffices to construct $(f_{ij})$ in the claim for these sufficiently large $3n^{\epsilon}+r$ instead of general $l$.
Indeed, given sufficiently large integer $N \not \equiv 0 \pmod 3$, take an integer $n$ satisfying 
\[3n^{\epsilon}+r < N \leq 3(n+1)^{\epsilon}+r,\]
where $N \equiv r \pmod 3$ and $0<r<3$.
If $(f_{ij})_{i \in [m+1], j \in [m]}$ is as claimed for $3(n+1)^{\epsilon}+r$, then applying a partial restriction given by a partial $3$-partition leaving elements in $[N]$ in the universe $[3(n+1)^{\epsilon}+r]$, we get $(f_{ij})_{i \in [m+1], j \in [m]}$ for $N$.
Since 
\[m=(3(n+1)^{\epsilon}+r)^{O(1)}=(3n^{\epsilon}+r)^{O(1)}=N^{O(1)},\]
 the claim for general $N$ follows.

Now, we focus on constructing $(f_{ij})$ in the claim for sufficiently large $3n^{\epsilon}+r$.
We renew the expression $3n^{\epsilon}+r$ and write it as (new) $n$ from now on for readability, refreshing our mind.

Recall that the evaluation $T$ is sound for Frege rules in $\pi_{n}$, and it satisfies $T \not \models Count^{3}_{n}$ at the same time.
It follows that some instance
\begin{align*}
 I:=oddtown_{m}[\sigma_{ij}/s_{ij}, \tau_{ij}/q_{ij}, \varphi_{ie}/p_{ie}, \psi_{ii^{\prime}e}/r_{ii^{\prime}e}]
\end{align*}
in $\pi_{n}$ satisfies $T \not \models I$.
 Restricting $O(1)$-elements more if necessary, we may assume that actually $br_{0}(T_{I}) = br(T_{I})$ holds.
 
 Clearly, $m = n^{O(1)}$.
For $i \in [m+1]$ and $j \in [m]$,
\begin{align*}
f_{ij} := \sum_{b \in br_{1}(T_{\sigma_{ij}})}\prod_{e \in b} x_{e}.
\end{align*}
We prove that these polynomials satisfy the required properties.

First, fix $i<i^{\prime} \in [m+1]$.
 We construct a $NS$-proof of $\sum_{j=1}^{m}f_{ij}f_{i^{\prime}j}=0$ over $\FF_{2}$ from the system $\lnot Count^{3}_{n}$．
For each $j \in [m]$, construct $U^{1}_{j}$ ($j \in [m]$) as follows:
 \begin{align*}
 U^{1}_{j}:=T_{\sigma_{ij}}*\sum_{b \in br_{1}(T_{\sigma_{ij}})} (T_{\sigma_{i^{\prime}j}}*\sum_{b^{\prime} \in br_{1}(T_{\sigma_{i^{\prime}j}})} (T_{i,i^{\prime},j})^{b^{\prime}})^{b},
\end{align*}
where
\begin{align*}
T_{i,i^{\prime},j}:=
T_{\lnot \sigma_{ij} \lor \lnot \sigma_{i^{\prime}j} \lor \bigvee_{e: j \in e \in [m]^{(2)}} \psi_{ii^{\prime}e}}.
\end{align*}
Now, define $U_{j}$ ($j \in [m]$) as follows:
\begin{itemize}
 \item Consider each branch $r\in br(U^{1}_{j})$ of the form $r=b(b^{\prime})^{b}d^{bb^{\prime}}$, where
 \begin{align*} 
 b \in br_{1}(T_{\sigma_{ij}}), b^{\prime} \in br_{1}(T_{\sigma_{i^{\prime}j}}), d\in br( T_{i,i^{\prime},j}).
 \end{align*}
 Let $S^{1}_{j}$ be the set of all branches of $U^{1}_{j}$ of the above form.
 Since $d || bb^{\prime}$ and
 \begin{align*}
 &T \models \lnot \sigma_{ij} \lor \lnot \sigma_{i^{\prime}j} \lor \bigvee_{e: j \in e \in [m]^{(2)}} \psi_{ii^{\prime}e}, \\
 &T \models \lnot \psi_{ii^{\prime}e} \lor \lnot \psi_{ii^{\prime}e^{\prime}} \quad (\forall e^{\prime} \perp e)
 \end{align*}
There is a unique $e_{r} \in [m]^{(2)}$ such that 
 \begin{itemize}
  \item $j \in e_{r}$, and
  \item $d$ is an extension of some $d^{\prime} \in br_{1}(T_{\psi_{ii^{\prime}e_{r}}})$.  
 \end{itemize}
 \item Let $j^{\prime}_{r,j}$ be the element of $e_{r}$ other than $j$ (that is, $e_{r}=\{j, j^{\prime}_{r,j}\}$).
 \item We define $U^{2}_{j} := U^{1}_{j} * \sum_{r \in S^{1}_{j}} (T_{\sigma_{ij^{\prime}_{r,j}}\lor \lnot \psi_{ii^{\prime}e_{r}}}*T_{\sigma_{i^{\prime}j^{\prime}_{r,j}}\lor \lnot \psi_{ii^{\prime}e_{r}}})^{r}$.
 \item Let $S^{2}_{j} \subseteq br(U^{2}_{j})$ be the set of all branches extending some element of $S^{1}_{j}$.
 \item For each 
 \begin{align*}
 u=rs^{r} \in S^{2}_{j} \quad (r \in S^{1}_{j}, s \in br (T_{\sigma_{ij^{\prime}_{r,j}}\lor \lnot \psi_{ii^{\prime}e_{r}}}*T_{\sigma_{i^{\prime}j^{\prime}_{r,j}}\lor \lnot \psi_{ii^{\prime}e_{r}}})),
 \end{align*}
  Let $J^{\prime}_{u,j} := j^{\prime}_{r,j}$.
 \item We set $U_{j}:=U^{2}_{j}*\sum_{u \in S^{2}_{j}} (U^{2}_{J^{\prime}_{u,j}})^{u}$. 
\end{itemize}
\quad For each $j \in [m]$, let $B_{j} \subseteq br(U_{j})$ be the set of branches extending some element of $S^{2}_{j}$.
Under $\lnot Count^{3}_{n}$, we see the equation 
\begin{align*}
\sum_{\alpha \in B_{j}} x_{\alpha}=f_{ij}f_{i^{\prime}j}
\end{align*}
 has a $NS$-proof over $\FF_{2}$ with degree $O(1)$.
Therefore, we obtain a $NS$-proof of 
\begin{align*}
 \sum_{j \in [m]} \sum_{\alpha \in B_{j}}x_{\alpha} = \sum_{j \in [m]} f_{ij}f_{i^{\prime}j}
\end{align*}
 over $\FF_{2}$ with degree $O(1)$.
The LHS is $0$ based on the following reasoning:
Let $b=uv^{u} \in B_{j}$, where $u \in S^{2}_{j}$ and $v \in br(U^{2}_{J^{\prime}_{u,j}})$. 
Let $l:=J^{\prime}_{u,j}$.
It is easy to see that $v \in S^{2}_{l}$, and $b \in B_{l}$.
Decomposing $b=cd^{c}$, where $c \in S_{l}$ and $d \in S_{J^{\prime}_{c,l}}$, we obtain $J^{\prime}_{c,l}=j$.
This argument gives a natural $2$-partition of the disjoint union $\bigsqcup_{j \in [m]} B_{j}$, pairing branches giving the same partial injections.
\\
 \quad Lastly, we prove that $\sum_{j=1}^{m} f_{ij}=1$ ($i \in [m+1]$) has a $NS$-proof from $\lnot Count^{3}_{n}$ with degree $O(1)$. 
 For $j \in [m]$, let
 \begin{align*}
  V_{j}:= T_{\sigma_{ij}}*\sum_{b \in br_{1}(T_{\sigma_{ij}})} (T_{i,j})^{b}.
 \end{align*}
 where $T_{i,j}:=T_{\lnot \sigma_{ij} \lor \tau_{ij} \lor \bigvee_{e: j \in e \in [n]^{(2)}} \varphi_{ie}}$.\\
\quad Let $B_{j}$ be the set of branches $b \in br(V_{j})$ extending some $b^{\prime} \in br_{1}(T_{\sigma_{ij}})$.
Since 
\begin{align*}
&T \models \lnot \tau_{ij} \lor \lnot \varphi_{ie} \quad (j \in e \in [n]^{(2)})\quad \mbox{and}\\
&T \models \lnot \varphi_{ie} \lor \lnot \varphi_{ie^{\prime}} \quad (e \perp e^{\prime}),
\end{align*}
 each $b \in B_{j}$ satisfies exactly one of the following:
 \begin{enumerate}
  \item\label{single} $b$ is an extension of some $b^{\prime} \in br_{1}(T_{\tau_{ij}})$.
  \item\label{double} There exists a unique $e \in [m]^{(2)}$ ($j \in e$) such that $b$ is an extension of some $b^{\prime} \in br_{1}(T_{\varphi_{ie}})$.
 \end{enumerate}
 Define $R_{b}$ as follows:
 \begin{enumerate}
  \item If the condition \ref{single} is satisfied, 
  \[R_{b}:= T_{i}*T_{\sigma_{ij} \lor \lnot \tau_{ij}},\]
   where $T_{i} := T_{\bigvee_{j}\tau_{ij}}$.
  \item If the condition \ref{double} is satisfied, and $\{j,j^{\prime}\}:=e$,
   then 
  \[R_{b}:=T_{\sigma_{ij^{\prime}} \lor \lnot \varphi_{ie}}*(V_{j^{\prime}}*T_{\sigma_{ij} \lor \lnot \varphi_{ie}}).\]
 \end{enumerate}
 Consider 
 \[W_{j}:=V_{j}*\sum_{b \in B_{j}}(R_{b})^{b}.\]
  Let $C_{j}$ be the set of all branches $r \in br(W_{j})$ of the form $r=bd^{b}$, where $b \in B_{j}$.
 The equation 
 \begin{align*}
 \sum_{r \in C_{j}}x_{r}=f_{ij}
 \end{align*}
  has a $NS$-proof over $\FF_{2}$ from $\lnot Count^{3}_{n}$ with degree $O(1)$.\\
\quad Now, we define $Q_{i}$ as follows: 
since
\begin{align*}
T \models \bigvee_{j} \tau_{ij}\quad \mbox{and}\quad  T \models \lnot \tau_{ij} \lor \lnot \tau_{ij^{\prime}},
\end{align*}
we see each branch $b \in br(T_{i})$ has the unique $j_{b}$ such that $b$ is an extension of $b^{\prime} \in br_{1}(T_{\tau_{ij_{b}}})$.
We set
\begin{align*}
Q_{i}:= T_{i}*\sum_{b \in br(T_{i})}(T_{\sigma_{ij_{b}} \lor \lnot \tau_{ij_{b}}} *(T_{\sigma_{ij_{b}}} * T_{i,j_{b}}))^{b}.
\end{align*}
 \quad We already know that there exists a $NS$-proof of $\sum_{\beta \in br(Q_{i})} x_{\beta}=1$ from $\lnot Count^{3}_{n}$ with degree $O(1)$.
 Observing
 \begin{align}\label{modout}
  \sum_{j} \sum_{r \in C_{j}} x_{r}+ \sum_{\beta \in br(Q_{i})} x_{\beta} = 0,
 \end{align}
 we obtain a $NS$-proof of $\sum_{j=1}^{m} f_{ij}=1$ from $\lnot Count^{3}_{n}$ with degree $O(1)$.\\
\quad The observation follows from a $2$-partition of $(\bigsqcup_{j}C_{j}) \sqcup br(Q_{i})$ similar to the one appearing in the proof of item 3. To be concrete, consider the following labeling of $r = bd^{b} \in C_{j}$ (where $b \in B_{j}$):
\begin{itemize}
 \item If $b$ satisfies the condition \ref{single} in the definition of $W_{j}$, label it by $\{ j \}$.
 \item If $b$ satisfies the condition \ref{double} in the definition of $W_{j}$, and $\{j,j^{\prime}\}:=e$, label $b$ by $e$.
\end{itemize}
To make $W_{j}$ fully labeled, we label each $b \in br(W_{j}) \setminus C_{j}$ by a symbol $\perp$.
Then we obtain:
\begin{itemize}
 \item For each $j \neq j^{\prime}$, $br_{\{  j,j^{\prime} \}} (W_{j}) = br_{\{ j,j^{\prime} \}} (W_{j^{\prime}})$.
 \item $\bigsqcup_{j \in [m]} br_{\{ j \}} (W_{j}) = br(Q_{i})$.
\end{itemize}
These give the equation (\ref{modout}).
\end{proof}

\section{On the strength of Fisher's inequality}\label{On the strength of Fisher's inequality}
\quad When we discuss whether the condition given in Theorem \ref{OnF2} actually holds or not, it is natural also to consider the $\KK$-analogue of the condition, where $\KK$ is an arbitrary field other than $\FF_{2}$.
The following combinatorial principle (see Remark \ref{intuition} for the informal meaning) relates to a condition similar to the analog.

\begin{defn}
We define the $\Sigma^{B}_{0}$ $\mathcal{L}^{2}_{A}$ formula $FIE(n,S,R)$ as follows:
\begin{align*}
&FIE(n,S,R) := \\
&\lnot \Bigg(
 \forall i \in [n+1] \exists j \in [n] S(i,j)\\ 
 & \land \forall i_{1}< i_{2} \in [n+1] \exists j \in [n] 
 ((S(i_{1},j) \land \lnot S(i_{2},j)) \lor (\lnot S(i_{1},j) \land S(i_{2},j) ))\\
 &\land \forall i_{1} < i_{2} \in [n+1] \forall i_{1}^{\prime} < i_{2}^{\prime} \in [n+1] \forall j \in [n]\\
  &\quad \quad (\lnot S(i_{1},j) \lor \lnot S(i_{2},j) \lor \exists j^{\prime} \in [n] R(i_{1},i_{2},i_{1}^{\prime}, i_{2}^{\prime},j,j^{\prime}) ) \\
 &\land\forall i_{1}<i_{2} \in [n+1] \forall i_{1}^{\prime} < i_{2}^{\prime} \in [n+1] \forall j^{\prime} \in [n] \\
 &\quad \quad (\lnot S(i_{1}^{\prime},j) \lor \lnot S(i_{2}^{\prime},j) \lor \exists j \in [n] R(i_{1},i_{2},i_{1}^{\prime}, i_{2}^{\prime},j,j^{\prime})) \\
 &\land \forall i_{1}<i_{2} \in [n+1] \forall i_{1}^{\prime} < i_{2}^{\prime} \in [n+1] \forall j, j^{\prime} \in [n] (\lnot R(i_{1},i_{2},i_{1}^{\prime}, i_{2}^{\prime},j,j^{\prime}) \lor S(i_{1},j) )\\
 &\land \forall i_{1}<i_{2} \in [n+1] \forall i_{1}^{\prime} < i_{2}^{\prime} \in [n+1] \forall j, j^{\prime} \in [n] (\lnot R(i_{1},i_{2},i_{1}^{\prime}, i_{2}^{\prime},j,j^{\prime}) \lor S(i_{2},j))\\
  &\land \forall i_{1}<i_{2} \in [n+1] \forall i_{1}^{\prime} < i_{2}^{\prime} \in [n+1] \forall j, j^{\prime} \in [n]( \lnot R(i_{1},i_{2},i_{1}^{\prime}, i_{2}^{\prime}, j,j^{\prime}) \lor S(i_{1}^{\prime},j^{\prime}) )\\
  &\land \forall i_{1}<i_{2} \in [n+1] \forall i_{1}^{\prime} < i_{2}^{\prime} \in [n+1] \forall j, j^{\prime} \in [n]( \lnot R(i_{1},i_{2},i_{1}^{\prime}, i_{2}^{\prime},j,j^{\prime}) \lor S(i_{2}^{\prime},j^{\prime})) \\
  &\land \forall i_{1}<i_{2} \in [n+1] \forall i_{1}^{\prime} < i_{2}^{\prime} \in [n+1] \forall j \in [n]\forall j^{\prime} \neq  j^{\prime\prime} \in [n]\\
  &\quad \quad  (\lnot R(i_{1},i_{2},i_{1}^{\prime}, i_{2}^{\prime},j,j^{\prime}) \lor \lnot R(i_{1},i_{2},i_{1}^{\prime}, i_{2}^{\prime},j,j^{\prime\prime}) )\\
  &\land \forall i_{1}<i_{2} \in [n+1] \forall i_{1}^{\prime} < i_{2}^{\prime}\in [n+1] \forall j^{\prime} \in [n]\forall j \neq  \widetilde{j} \in [n] \\
  &\quad \quad (\lnot R(i_{1},i_{2},i_{1}^{\prime}, i_{2}^{\prime},j,j^{\prime}) \lor \lnot R(i_{1},i_{2},i_{1}^{\prime}, i_{2}^{\prime},\widetilde{j},j^{\prime}))
 \Bigg)
\end{align*}
Furthermore, we define the propositional formula $FIE_{n}$ as follows:
  \begin{align*}
  FIE_{n}:=
&\lnot \Bigg(
 \bigwedge_{i \in [n+1]} \bigvee_{j \in [n]} s_{ij} \\
 &\land \bigwedge_{i_{1}<i_{2} \in [n+1]} \bigvee_{j \in [n]} ((s_{i_{1}j} \land \lnot s_{i_{2}j}) \lor (\lnot s_{i_{1}j} \land s_{i_{2}j})) \\
 &\land\bigwedge_{i_{1}<i_{2} \in [n+1]} \bigwedge_{i_{1}^{\prime} < i_{2}^{\prime} \in [n+1]} \bigwedge_{j \in [n]} (\lnot s_{i_{1}j} \lor \lnot s_{i_{2}j} \lor \bigvee_{j^{\prime} \in [n]} r^{i_{1},i_{2},i_{1}^{\prime}, i_{2}^{\prime}}_{j,j^{\prime}} )\\
&\land\bigwedge_{i_{1}<i_{2} \in [n+1]} \bigwedge_{i_{1}^{\prime} < i_{2}^{\prime} \in [n+1]} \bigwedge_{j^{\prime} \in [n]} (\lnot s_{i_{1}^{\prime}j} \lor \lnot s_{i_{2}^{\prime}j} \lor \bigvee_{j \in [n]} r^{i_{1},i_{2},i_{1}^{\prime}, i_{2}^{\prime}}_{j,j^{\prime}}) \\
 &\land \bigwedge_{i_{1}<i_{2} \in [n+1]} \bigwedge_{i_{1}^{\prime} < i_{2}^{\prime} \in [n+1]} \bigwedge_{j, j^{\prime} \in [n]} (\lnot r^{i_{1},i_{2},i_{1}^{\prime}, i_{2}^{\prime}}_{j,j^{\prime}} \lor s_{i_{1}j} )\\
 &\land \bigwedge_{i_{1}<i_{2} \in [n+1]} \bigwedge_{i_{1}^{\prime} < i_{2}^{\prime} \in [n+1]} \bigwedge_{j, j^{\prime} \in [n]} (\lnot r^{i_{1},i_{2},i_{1}^{\prime}, i_{2}^{\prime}}_{j,j^{\prime}} \lor s_{i_{2}j})\\
  &\land \bigwedge_{i_{1}<i_{2} \in [n+1]} \bigwedge_{i_{1}^{\prime} < i_{2}^{\prime} \in [n+1]} \bigwedge_{j, j^{\prime} \in [n]}( \lnot r^{i_{1},i_{2},i_{1}^{\prime}, i_{2}^{\prime}}_{j,j^{\prime}} \lor s_{i_{1}^{\prime}j^{\prime}} )\\
  &\land \bigwedge_{i_{1}<i_{2} \in [n+1]} \bigwedge_{i_{1}^{\prime} < i_{2}^{\prime} \in [n+1]} \bigwedge_{j, j^{\prime} \in [n]}( \lnot r^{i_{1},i_{2},i_{1}^{\prime}, i_{2}^{\prime}}_{j,j^{\prime}} \lor s_{i_{2}^{\prime}j^{\prime}}) \\
  &\land \bigwedge_{i_{1}<i_{2} \in [n+1]} \bigwedge_{i_{1}^{\prime} < i_{2}^{\prime} \in [n+1]} \bigwedge_{j \in [n]}\bigwedge_{j^{\prime} \neq  j^{\prime\prime} \in [n]} (\lnot r^{i_{1},i_{2},i_{1}^{\prime}, i_{2}^{\prime}}_{j,j^{\prime}} \lor \lnot r^{i_{1},i_{2},i_{1}^{\prime}, i_{2}^{\prime}}_{j,j^{\prime\prime}} ) \\
  &\land \bigwedge_{i_{1}<i_{2} \in [n+1]} \bigwedge_{i_{1}^{\prime} < i_{2}^{\prime} \in [n+1]} \bigwedge_{j^{\prime} \in [n]}\bigwedge_{j \neq  \widetilde{j} \in [n]} (\lnot r^{i_{1},i_{2},i_{1}^{\prime}, i_{2}^{\prime}}_{j,j^{\prime}} \lor \lnot r^{i_{1},i_{2},i_{1}^{\prime}, i_{2}^{\prime}}_{\widetilde{j},j^{\prime}})
 \Bigg)
  \end{align*}

\end{defn}

\begin{rmk}\label{intuition}
The above formulae are formalizations of Fisher's inequality: there does not exist a family $\{S_{i}\}_{i \in [n+1]}$ satisfying the following:
\begin{itemize}
 \item For each $i$, $\emptyset \neq S_{i} \subseteq [n]$.
 \item For each $i_{1}<i_{2}$, $S_{i_{1}} \neq S_{i_{2}}$.
 \item For each $i_{1}<i_{2}$ and $i^{\prime}_{1}<i^{\prime}_{2}$, $\#(S_{i_{1}} \cap S_{i_{2}})=\#(S_{i^{\prime}_{1}} \cap S_{i^{\prime}_{2}})$.
\end{itemize} 

In our definition of $FIE(n,S,R)$, $S$ intuitively gives a family $\{S_{i}\}_{i \in [n+1]}$, and $R$ gives a family of bijections 
\begin{align*}
\{R^{i_{1},i_{2},i^{\prime}_{1},i^{\prime}_{2}} \colon S_{i_{1}} \cap S_{i_{2}} \rightarrow S_{i^{\prime}_{1}} \cap S_{i^{\prime}_{2}}\}_{i_{1}<i_{2} \& i^{\prime}_{1}<i^{\prime}_{2}}.
\end{align*}
\end{rmk}

\begin{rmk}\label{ondualformalizationofFisher'sinequality}
In \cite{Fisher's inequality}, Fisher gave a special case of the statement, and later, it was generalized to the above form by \cite{Today's Fisher's inequality} and others. 
See introductions of \cite{Today's Fisher's inequality} and \cite{A note on Fisher's inequality} for historical remarks.
Note that there are several versions of the statement of Fisher's inequality.
Primarily, it is often presented in a dual form, switching the roles of sets and elements. 
The form we adopted is following the presentation of \cite{Hard}.
To verify the above version, see the proof of Theorem 4 in \cite{Hard}.
\end{rmk}

Although Fisher's inequality was proposed as a candidate for potentially hard tautology for the Frege system, it is provable in $\VTC^{0}$.
I was unaware of this fact when this paper was submitted, and I would like to thank the anonymous referee for sharing the work of \cite{A note on Fisher's inequality}.
Around the very time \cite{Hard} was published, Woodall gave an elementary proof of Fisher's inequality in \cite{A note on Fisher's inequality}, which was already sufficient to construct short Frege proofs of Fisher's inequality.
The proof can be interpreted as a proof in $\VTC^{0}$ augmented with iterated multiplication of integers (or, equivalently, rationals).
Recently, Je\v{r}abek showed that $\VTC^{0}$ can formalize iterated multiplication of integers and prove its basic properties (\cite{VTC^{0}caniteratemultiplication}).
Since we have nothing to add to their results, we just present the claim for reference:
\begin{theorem}[Essentially by \cite{A note on Fisher's inequality} and \cite{VTC^{0}caniteratemultiplication}]
\[VTC^{0} \vdash FIE(n,S,R).\]
\end{theorem}
Note that \cite{A note on Fisher's inequality} takes the dual approach mentioned in Remark \ref{ondualformalizationofFisher'sinequality}, and also we do not need the ``nontriviality''-condition assumed in the article.

Now, we focus on the lower bounds of the strength of $FIE_{n}$.
It is easy to see that $FIE_{n}$ is a generalization of the pigeonhole principle.

\begin{proposition}
 $V^{0}+FIE_{k} \vdash injPHP^{n+1}_{n}$.
 Hence, for each $p \geq 2$, 
 \[V^{0} + Count^{p}_{k} \not \vdash FIE_{n}.\]
\end{proposition}
\begin{proof}
Argue in $V^{0}$. Suppose $f$ is an injection from $[n+1]$ to $[n]$. 
We set
$S_{i}:= \{f(i)\}$ ($i \in [n+1]$). Then, the $S_{i}$'s are distinct and $S_{i} \cap S_{j}=\emptyset$ for every $i<j$.
Hence, $FIE_{n}$ is violated.

 The latter part follows immediately from Theorem \ref{CountpinjPHP}.
\end{proof}

It is natural to ask the following:
\begin{question}
 Which $p$ satisfies $V^{0}+FIE_{k} \vdash Count^{p}_{n}$?
\end{question}
We conjecture that there is no such $p$. The following theorem gives a criterion for proving this conjecture.

\begin{theorem}
 Let $\KK$ be a field. 
 Suppose $F_{d} + FIE_{k} \vdash_{poly(n)} Count^{p}_{n}$. 
 Then, for large enough $n \not \equiv0 \pmod p$, there exist $m =n^{O(1)}$ and families $(f_{ij})_{i \in [m+1], j \in [m]}$, $(a_{ij})_{i \in [m+1], j \in [m]}$ and $(b_{ii^{\prime}j})_{i<i^{\prime} \in [m+1], j \in [m]}$ of polynomials over $\KK$ such that the following equalities have $NS$-proofs over $\KK$ from $\lnot Count^{p}_{n}$ with degree $O(1)$:

  \begin{align*}
  \sum_{j=1}^{m}f_{i_{1}j}f_{i_{2}j}&=\sum_{j=1}^{m}f_{i_{1}^{\prime}j}f_{i_{2}^{\prime}j} \quad &(i_{1} < i_{2}\in [m+1] \ \& \ i_{1}^{\prime} < i_{2}^{\prime} \in [m+1]),\\
   a_{ij}(1-f_{ij})&=0 \quad &(i \in [m+1]),\\
   \sum_{j=1}^{m}a_{ij}&=1 \quad &(i \in [m+1]),\\
  b_{ii^{\prime}j}f_{ij}f_{i^{\prime}j}&=0 \quad &(i<i^{\prime} \in [m+1], \ \& \ j \in [m]),\\
  b_{ii^{\prime}j}(1-f_{ij})(1-f_{i^{\prime}j})&=0 \quad &(i<i^{\prime} \in [m+1], \ \& \ j \in [m]).\\
  \sum_{j=1}^{m}b_{ii^{\prime}j}&=1 \quad &(i<i^{\prime} \in [m+1]).
  \end{align*}
   
\end{theorem}

\begin{proof}
For readability, we assume $p=3$.
 Assume proofs $(\pi_{n})_{n \in \NN}$ witness
 \begin{align*}
  F_{d} + FIE_{k} \vdash_{poly(n)} Count^{3}_{n}.
 \end{align*} 
 We focus on $n \not \equiv 0 \pmod 3$.
 Let $\Gamma_{n}$ be the set of subformulae of $\pi_{n}$.
 As the proof of Theorem \ref{OnF2}, using the switching lemma for $3$-tree and a padding argument, we may assume that there exists an $O(1)$-evaluation $T^{n}$ of $\Gamma_{n}$, and it suffices to construct $(f_{ij})$, $(a_{ij})$ and $(b_{ii^{\prime}j})$ for $n$.
 We fix a large enough $n \not \equiv 0 \pmod 3$, and suppress scripts $n$ of $T^{n}$, $\rho_{n}$, etc.
 
 Since $T \not \models Count^{3}_{n}$,
 soundness gives that some instance 
 \begin{align*}
 I:= FIE_{m}[\sigma_{ij} / s_{ij}, \varphi^{i_{1},i_{2},i_{1}^{\prime}, i_{2}^{\prime}}_{j,j^{\prime}}/r^{i_{1},i_{2},i_{1}^{\prime}, i_{2}^{\prime}}_{j,j^{\prime}}]
 \end{align*}
 in $\pi_{n}$ satisfies $T \not \models I$. 
 With an additional restriction, we may assume that \[br_{0}(T_{I}) = br(T_{I}).\]\\
 \quad We obtain the following：
 \begin{enumerate}
  \item Let $T_{i}:=T_{\bigvee_{j \in [m]} \sigma_{ij}}$.
  Since 
   \begin{align*}
 T \models \bigvee_{j \in [m]} \sigma_{ij},
   \end{align*}
    each $b \in br(T_{i})$ has at least one $j_{b} \in [m]$ and $b^{\prime} \in br_{1}(T_{\sigma_{ij_{b}}})$ such that $b^{\prime} \subseteq b$.
    We relabel each branch $b \in br(T_{i})$ with $\langle j_{b} \rangle$ and obtain a labeled $injPHP$-tree $\widetilde{T}_{i}$.
    
  \item Let $T_{i_{1}, i_{2}} := T_{\bigvee_{j \in [m]} ((\sigma_{i_{1}j} \land \lnot \sigma_{i_{2}j}) \lor (\lnot \sigma_{i_{1}j} \land \sigma_{i_{2}j}))}$. Since 
 \begin{align*}
 T \models \bigvee_{j \in [m]} ((\sigma_{i_{1}j} \land \lnot \sigma_{i_{2}j}) \lor (\lnot \sigma_{i_{1}j} \land \sigma_{i_{2}j})),
 \end{align*}
  each $b \in br(T_{i_{1},i_{2}})$ has at least one $j_{b}$ satisfying one of the following：
  \begin{enumerate}
   \item For all $b^{\prime} \in br_{0}(T_{\sigma_{i_{1}j_{b}}})\cup br_{1}(T_{\sigma_{i_{2}j_{b}}})$, $b\perp b^{\prime}$．
   \item For all $b^{\prime} \in br_{1}(T_{\sigma_{i_{1}j_{b}}})\cup br_{0}(T_{\sigma_{i_{2}j_{b}}})$, $b \perp b^{\prime}$．
   \end{enumerate}
 We relabel each branch $b \in br(T_{i_{1}, i_{2}})$ with $\langle j_{b} \rangle$ and obtain a labeled $injPHP$-tree $\widetilde{T}_{i_{1},i_{2}}$.
 
 \item Let $ T^{i_{1},i_{2},i_{1}^{\prime}, i_{2}^{\prime}}_{1,j} := T_{ \lnot \sigma_{i_{1}j} \lor \lnot \sigma_{i_{2}j} \lor \bigvee_{j^{\prime} \in [m]} \varphi^{i_{1},i_{2},i_{1}^{\prime}, i_{2}^{\prime}}_{j,j^{\prime}}}$.
 Each $b \in br(T^{i_{1},i_{2},i_{1}^{\prime}, i_{2}^{\prime}}_{1,j} )$ is an extension of some element of $br_{0}(T_{\sigma_{i_{1}j}}), br_{0} (T_{\sigma_{i_{2}j}}), \bigcup_{j^{\prime}}br_{1}(T_{\varphi^{i_{1},i_{2},i_{1}^{\prime}, i_{2}^{\prime}}_{j,j^{\prime}}})$.
 If $b$ is an extension of an element of $br_{1}(T_{\varphi^{i_{1},i_{2},i_{1}^{\prime}, i_{2}^{\prime}}_{j,j^{\prime}}})$, such $j^{\prime}$ is unique.
 \item Let $ T^{i_{1},i_{2},i_{1}^{\prime}, i_{2}^{\prime}}_{2,j^{\prime}} := T_{\lnot \sigma_{i_{1}^{\prime}j} \lor \lnot \sigma_{i_{2}^{\prime}j} \lor \bigvee_{j \in [m]} \varphi^{i_{1},i_{2},i_{1}^{\prime}, i_{2}^{\prime}}_{j,j^{\prime}}}$. 
 Each $b \in br(T^{i_{1},i_{2},i_{1}^{\prime}, i_{2}^{\prime}}_{2,j^{\prime}})$ is an extension of an element of $br_{0}(T_{\sigma_{i_{1}^{\prime}j^{\prime}}}), br_{0} (T_{\sigma_{i_{2}^{\prime}j^{\prime}}}), \bigcup_{j}br_{1}(T_{r^{i_{1},i_{2},i_{1}^{\prime}, i_{2}^{\prime}}_{j,j^{\prime}}})$.
 If $b$ is an extension of an element of $br_{1}(T_{r^{i_{1},i_{2},i_{1}^{\prime}, i_{2}^{\prime}}_{j,j^{\prime}}})$, such $j$ is unique.
  \end{enumerate}
  Now, we set 
  \begin{align*}
  f_{ij}&:= \sum_{\alpha \in br_{1}(T_{\sigma_{ij}})}x_{\alpha},\\
  a_{ij}&:= \sum_{\alpha \in br_{\langle j \rangle}(\widetilde{T}_{i})}x_{\alpha}\\
  b_{i_{1}i_{2}j}&:=\sum_{\alpha \in br_{\langle j \rangle}(\widetilde{T}_{i_{1},i_{2}})}x_{\alpha}.\\
  &( i \in [m+1], j \in [m])
  \end{align*}
  Clearly, $m \leq n^{O(1)}$.\\
  \quad We show that each of the following has a $NS$-proof from $\lnot Count^{3}_{n}$ over $\KK$ with $O(1)$-degree：
  \begin{align}
   &\sum_{j=1}^{m} a_{ij}=1, \label{nonzeroa}\\
   &\sum_{j=1}^{m} b_{i_{1}i_{2}j} = 1, \label{nonzero}\\
   &\sum_{j=1}^{m}f_{i_{1}j}f_{i_{2}j} = \sum_{j=1}^{m} f_{i_{1}^{\prime}j}f_{i_{2}j^{\prime}}, \label{innersame}\\
   &a_{ij}(1-f_{ij})=0, \label{existence} \\
   &b_{i_{1}i_{2}j}f_{i_{1}j}f_{i_{2}j}=0 \label{eitherzero},\\
   &b_{i_{1}i_{2}j}(1-f_{i_{1}j})(1-f_{i_{2}j})=0 \label{eitherone}.
  \end{align}
  \centerline{$(i, i_{1}, i_{2}, i^{\prime}_{1}, i^{\prime}_{2} \in [m+1] \& i_{1}< i_{2} \& i^{\prime}_{1}<i^{\prime}_{2} \& j \in [m] )$}
\begin{enumerate}
   \item [(\ref{nonzeroa}):]  Since the left-hand side is the sum of all branches of the $3$-partition tree $\widetilde{T}_{i}$．
   \item[(\ref{nonzero}):]  Since the left-hand side is the sum of all branches of the $3$-partition tree $\widetilde{T}_{i_{1},i_{2}}$．
   \item[(\ref{innersame}):]  
   We first define $A_{i_{1},i_{2},j}:= T_{\sigma_{i_{1}j}} * T_{\sigma_{i_{2}j}}$ ($i_{1}, i_{2} \in [m+1], j \in [m], i_{1}<i_{2}$).
   Let $B_{i_{1},i_{2},j}$ be the set of all branches $b \in A_{i_{1},i_{2},j}$ having the form
   \begin{align*}
   b=cd^{c} \quad (c \in br_{1}(T_{\sigma_{i_{1}j}}), d \in br_{1}(T_{\sigma_{i_{2}j}})).
   \end{align*} 
   It is easy to construct a $NS$-proof of
   \begin{align}
    f_{i_{1}j}f_{i_{2}j} = \sum_{b \in B_{i_{1},i_{2},j}} x_{b} \label{andeq}
   \end{align}
   from $\lnot Count^{3}_{n}$ over $\KK$ with degree $O(1)$.\\
   \quad Now, fix $i_{1}, i_{2}, i^{\prime}_{1}, i^{\prime}_{2} \in [m+1]$ such that $i_{1}< i_{2}$ and $i^{\prime}_{1}<i^{\prime}_{2}$. 
   For each $j \in [m]$, consider the trees 
   \begin{align*}
    R_{j} &:= A_{i_{1},i_{2},j}*\sum_{b \in B_{i_{1},i_{2},j}}(T^{i_{1},i_{2},i_{1}^{\prime}, i_{2}^{\prime}}_{1,j})^{b}\\
    R^{\prime}_{j}&:=A_{i^{\prime}_{1},i^{\prime}_{2},j}*\sum_{b \in B_{i^{\prime}_{1},i^{\prime}_{2},j}}(T^{i_{1},i_{2},i_{1}^{\prime}, i_{2}^{\prime}}_{2,j})^{b}.
   \end{align*}
   For each $r=bd^{b}$ ($b \in B_{i_{1},i_{2},j}$, $d \in br(T^{i_{1},i_{2},i_{1}^{\prime}, i_{2}^{\prime}}_{1,j})$), since $d || b$, there exists a unique $j^{\prime}_{r}$ such that $d$ is an extension of some $c \in br_{1}(T_{\varphi^{i_{1},i_{2},i_{1}^{\prime}, i_{2}^{\prime}}_{j,j^{\prime}_{r}}})$. 
   Let $B_{j} \subseteq br(R_{j})$ be the set of all branches having the above form.\\
   \quad Similarly, for each $r^{\prime}=bd^{b}$ ($b \in B_{i^{\prime}_{1},i^{\prime}_{2},j}$, $d \in br(T^{i_{1},i_{2},i_{1}^{\prime}, i_{2}^{\prime}}_{2,j})$), since $d || b$, there exists a unique $\widehat{j}_{r^{\prime}}$ such that $d$ is an extension of some $c \in br_{1}(T_{\varphi^{i_{1},i_{2},i_{1}^{\prime}, i_{2}^{\prime}}_{\widehat{j}_{r^{\prime}},j}})$. 
   Let $B^{\prime}_{j} \subseteq br(R^{\prime}_{j})$ be the set of all branches having the above form.\\
   \quad Now, we define
   \begin{align*}
   T_{j,j^{\prime}}:= 
   (((T_{\lnot r^{i_{1},i_{2},i_{1}^{\prime}, i_{2}^{\prime}}_{j,j^{\prime}} \lor s_{i_{1}j}} * 
   T_{\lnot r^{i_{1},i_{2},i_{1}^{\prime}, i_{2}^{\prime}}_{j,j^{\prime}} \lor s_{i_{2}j} })*
   T_{\lnot r^{i_{1},i_{2},i_{1}^{\prime}, i_{2}^{\prime}}_{j,j^{\prime}} \lor s_{i^{\prime}_{1}j^{\prime}} })*
   T_{\lnot r^{i_{1},i_{2},i_{1}^{\prime}, i_{2}^{\prime}}_{j,j^{\prime}} \lor s_{i^{\prime}_{2}j^{\prime}} }
   ).
   \end{align*}
   for each $j\neq j^{\prime} \in [m]$. Using these trees, we define
   \begin{align*}
   S_{j} &:= R_{j} * \sum_{r \in B_{j}} (T_{j,j^{\prime}_{r}} * \sum_{t \in br(T_{j,j^{\prime}_{r}})} (R^{\prime}_{j^{\prime}_{r}})^{t})^{r},\\
   S^{\prime}_{j} &:= R^{\prime}_{j} * \sum_{r^{\prime} \in B^{\prime}_{j}} (T_{\widehat{j}_{r^{\prime}},j}*\sum_{t \in br(T_{\widehat{j}_{r^{\prime}},j})} (R_{\widehat{j}_{r^{\prime}}})^{t} )^{r^{\prime}}.
   \end{align*} 
   Label each branch $b \in br(S_{j})$ as follows:
   \begin{itemize}
   \item If $b$ extends some $r \in B_{j}$, then label $b$ with $\langle j,j^{\prime}_{r} \rangle$.
   \item Otherwise, label $b$ with the symbol $\perp$.
   \end{itemize} 
   Similarly, we label each branch $b^{\prime} \in br(S^{\prime}_{j})$ as follows:
   \begin{itemize}
    \item If $b$ extends some $r^{\prime} \in B^{\prime}_{j}$, then label $b$ with $\langle \widehat{j}_{r^{\prime}}, j\rangle$.
    \item Otherwise, label $b$ with the symbol $\perp$.
   \end{itemize}
   It is easy to see that for each $j,j^{\prime}$, $br_{\langle j,j^{\prime}\rangle} (S_{j})= br_{\langle j,j^{\prime} \rangle}(S^{\prime}_{j^{\prime}})$.
   Hence,
   \begin{align*}
    \sum_{j, j^{\prime} \in [m]} \sum_{\alpha \in br_{\langle j,j^{\prime} \rangle}(S_{j})} x_{\alpha}
    =
    \sum_{j, j^{\prime} \in [m]} \sum_{\beta \in br_{\langle j,j^{\prime} \rangle}(S^{\prime}_{j^{\prime}})} x_{\beta}.
   \end{align*}
   Furthermore, it is easy to see that the following have $NS$-proofs from $\lnot Count^{3}_{n}$ over $\KK$ with $O(1)$-degree:
   \begin{align*}
    &\sum_{ j^{\prime} \in [m]} \sum_{\alpha \in br_{\langle j,j^{\prime} \rangle}(S_{j})} x_{\alpha}
    = \sum_{b \in B_{i_{1},i_{2},j}} x_{b} \quad (j \in [m]),\\
    &\sum_{j \in [m]} \sum_{\beta \in br_{\langle j,j^{\prime} \rangle}(S^{\prime}_{j^{\prime}})} x_{\beta}
    = \sum_{b \in B_{i^{\prime}_{1},i^{\prime}_{2},j^{\prime}}} x_{b} \quad (j^{\prime} \in [m]).
   \end{align*}
   Hence, combined with (\ref{andeq}), they give a $NS$-proof of $\sum_{j} f_{i_{1}j}f_{i_{2}j} = \sum_{j^{\prime}}f_{i^{\prime}_{1}}f_{i^{\prime}_{2}}$ satisfying the required conditions.
   
   \item[(\ref{existence}):] It follows similarly as (\ref{eitherone}) below.
   
   \item[(\ref{eitherzero}):] $b_{i_{1}i_{2}j}f_{i_{1}j}f_{i_{2}j}=0$ follows easily from $\lnot Count^{3}_{n}$ since each $\alpha \in br_{\langle j \rangle} (\widetilde{T}_{i_{1},i_{2}})$ satisfies $\alpha \perp b$ for all $b \in B_{i_{1},i_{2},j}$.
   \item[(\ref{eitherone}):] Note that we have $NS$-proofs of the following:
   \begin{align*}
   &f_{i_{1}j}+\sum_{\beta \in br_{0}(T_{\sigma_{i_{1}j}})} x_{\beta}=1,\\
   &f_{i_{2}j}+\sum_{\beta \in br_{0}(T_{\sigma_{i_{2}j}})} x_{\beta}=1.\\
   &(j \in [m])
   \end{align*}
   Hence, $b_{i_{1}i_{2}j}(1-f_{i_{1}j})(1-f_{i_{2}j})=0$ follows easily from $\lnot Count^{3}_{n}$ by a similar reason as the previous item.
 \end{enumerate}
\end{proof}

\section*{Acknowledgements}
\quad The author deeply appreciates the anonymous referee's pointing out several mistakes in the initial version of this paper, numerous constructive comments and suggestions, and drawing my attention to \cite{A note on Fisher's inequality}.
The author is also grateful to Toshiyasu Arai for the helpful discussions and his patient support and to Satoru Kuroda for guiding me through the series of work on proof complexity of linear algebra.
In addition, the author thanks to Jan Kraj\'{i}\v{c}ek, Mykyta Narusevych, Ond\v{r}ej Je\v{z}il, Emil Je\v{r}abek, Pavel Pudl\'{a}k, Erfan Khaniki, and Moritz M\"{u}ller for their valuable comments on the technical barrier of resolving $ontoPHP$ v.s. $injPHP$ over $V^{0}$, which are closely related to Remark \ref{switchingforinjPHPisdifficult}, and pointing out several typos in the initial version of this paper.
This research was done during 2020-2022 and supported by:
\begin{itemize}
\item FoPM, WINGS Program, the University of Tokyo, and
\item the Research Institute for Mathematical Sciences,
an International Joint Usage/Research Center located in Kyoto University.
\end{itemize}
The author got the above feedback in 2023.

\end{document}